\documentclass[reqno]{amsart}

\usepackage{amscd,amsthm,amsmath,amssymb,mathtools,bm,float}
\usepackage{pgfplots, pgfplotstable}
\usepackage{verbatim}
\usepackage{xcolor}
\usepackage{url}
\usepackage{enumerate,enumitem}
\usepackage{graphicx}
\usepackage{algorithm}
\usepackage{algorithmic}
\usepackage{epstopdf,epsfig,subfigure}
\usepackage{curve2e}
\usepackage{mathrsfs}
\usepackage{array}
\usepackage{stackrel}

\newcolumntype{M}[1]{>{\centering\arraybackslash}m{#1}}

\usepackage[numbers,sort&compress]{natbib}

\usepackage{setspace}
\usepackage[
    bookmarks=true,         % show bookmarks bar?
    unicode=false,          % non-Latin characters in Acrobats bookmarks
    pdftoolbar=true,        % show Acrobats toolbar?
    pdfmenubar=true,        % show Acrobats menu?
    pdffitwindow=false,     % window fit to page when opened
    pdfstartview={FitH},    % fits the width of the page to the window
    pdftitle={Stabilization of Cahn--Hilliard equations},    % title
    pdfauthor={Author},     % author
    pdfsubject={Subject},   % subject of the document
    pdfcreator={Creator},   % creator of the document
    pdfproducer={Producer}, % producer of the document
    pdfkeywords={keyword1, key2, key3}, % list of keywords
    pdfnewwindow=true,      % links in new PDF window
    colorlinks=true,       % false: boxed links; true: colored links
    linkcolor=red,          % color of internal links (change box color with linkbordercolor)
    citecolor=blue,        % color of links to bibliography
    filecolor=magenta,      % color of file links
    urlcolor=cyan,           % color of external links
hypertexnames=false
]{hyperref}

\theoremstyle{plain}

\newtheorem{theorem}{Theorem}[section]

\newtheorem{lemma}[theorem]{Lemma}
\newtheorem{corollary}[theorem]{Corollary}

\newtheorem{assumption}[theorem]{Assumption}

\theoremstyle{definition}

\newtheorem{remark}[theorem]{Remark}

\numberwithin{equation}{section}

\usepackage{geometry}
\newcommand{\linspan}{\mathop{\rm span}\nolimits}

\newcommand{\rest}{\left.\kern-2\nulldelimiterspace\right|_}
\newcommand{\norm}[2]{\left|#1\right|_{#2}}
\newcommand{\dnorm}[2]{\left\|#1\right\|_{#2}}

\newcommand{\zero}{{\mathbf0}}
\newcommand{\Id}{{\mathbf1}}
\newcommand{\indf}{1}

\newcommand{\p}{\partial}

\newcommand*{\Bigcdot}{\raisebox{-.25ex}{\scalebox{1.25}{$\cdot$}}}

\newcommand{\clC}{{\mathcal C}}

\newcommand{\clE}{{\mathcal E}}
\newcommand{\clF}{{\mathcal F}}
\newcommand{\clG}{{\mathcal G}}
\newcommand{\clH}{{\mathcal H}}

\newcommand{\clL}{{\mathcal L}}

\newcommand{\clN}{{\mathcal N}}

\newcommand{\clQ}{{\mathcal Q}}

\newcommand{\clU}{{\mathcal U}}
\newcommand{\clV}{{\mathcal V}}

\newcommand{\clZ}{{\mathcal Z}}

\newcommand{\bbN}{{\mathbb N}}

\newcommand{\bbR}{{\mathbb R}}

\newcommand{\bbT}{{\mathbb T}}

\newcommand{\bfA}{{\mathbf A}}

\newcommand{\bfF}{{\mathbf F}}

\newcommand{\bfK}{{\mathbf K}}

\newcommand{\bfN}{{\mathbf N}}

\newcommand{\bfP}{{\mathbf P}}
\newcommand{\bfQ}{{\mathbf Q}}
\newcommand{\bfR}{{\mathbf R}}
\newcommand{\bfS}{{\mathbf S}}

\newcommand{\bfW}{{\mathbf W}}

\newcommand{\fkH}{{\mathfrak H}}

\newcommand{\fkR}{{\mathfrak R}}

\newcommand{\fkT}{{\mathfrak T}}

\newcommand{\fkW}{{\mathfrak W}}

\newcommand{\rmD}{{\mathrm D}}

\newcommand{\bfn}{{\mathbf n}}

\newcommand{\bfu}{{\mathbf u}}

\newcommand{\rmc}{{\mathrm c}}
\newcommand{\rmd}{{\mathrm d}}
\newcommand{\rme}{{\mathrm e}}

\newcommand{\fkp}{{\mathfrak p}}

\newcommand{\ttc}{{\tt c}}

\newcommand{\ttr}{{\tt r}}

\newcommand{\ovlineC}[1]{\overline C_{\left[#1\right]}}

\definecolor{DarkBlue}{rgb}{0,0.08,0.45}
\definecolor{DarkRed}{rgb}{.65,0,0}
\definecolor{applegreen}{rgb}{0.55, 0.71, 0.0}

\newcounter{mymac@matlab}
\newcommand{\matlab}{MATLAB%
   \ifnum\value{mymac@matlab}<1%
   \textregistered%
   \setcounter{mymac@matlab}{1}%
   \fi%
  }

\begin{document}
\title{Stabilization to trajectories of nonisothermal Cahn--Hilliard equations}
\author{Behzad Azmi$^{\tt1}$}
\author{Marvin Fritz$^{\tt2}$}
\author{S\'ergio S.~Rodrigues$^{\tt2}$}
\thanks{
\vspace{-1em}\newline\noindent
\newline\noindent
$^{\tt1}$ Dep.  Math. Stat., Univ.
Konstanz, Universit\"atstr. 10, D-78457 Konstanz, Germany.
  \newline\noindent
$^{\tt2}$ Johann Radon Inst. Comput. Appl. Math.,
  \"OAW, Altenbergerstr. 69, 4040 Linz, Austria.
\quad
\newline\noindent
{\sc Emails}:
{\small\tt behzad.azmi@uni-konstanz.de,\quad marvin.fritz@ricam.oeaw.ac.at,
\newline\hspace*{3.5em}sergio.rodrigues@ricam.oeaw.ac.at}%
}

\begin{abstract}
In this work, it is proven the semiglobal exponential stabilization to time-dependent trajectories of the nonisothermal Cahn--Hilliard equations. In the model, the input controls are given by explicit feedback operators that involve appropriate oblique projections. The actuators are given by a finite number of indicator functions. The results also hold for the isothermal Cahn--Hilliard system. Numerical simulations are shown that illustrate the stabilizing performance of the proposed input feedback operators.
\end{abstract}

\maketitle

\pagestyle{myheadings} \thispagestyle{plain} \markboth{\sc  B. Azmi, M. Fritz,  S.S.
Rodrigues}{\sc Stabilization of nonisothermal Cahn--Hilliard systems}

%%%%%%%%%%%%%%%%%%%%%%%%%%
%%%%%%%%%%%%%%%%%%%%%%%%%%
%%%%%%%%%%%%%%%%%%%%%%%%%%
\section{Introduction}
Let us be given a trajectory~$y_\ttr$ of the nonisothermal Cahn--Hilliard system
\begin{subequations}\label{sys-CH}
  \begin{align}  
 &\tfrac{\p}{\p t}y_{1\ttr} +  \nu_2\Delta^2 y_{1\ttr}  - \Delta\Bigl(f(y_{1\ttr})-\nu_1 y_{2\ttr}\Bigr)+g(y_{1\ttr})=h_1,\\
  &\tfrac{\p}{\p t}(y_{2\ttr} +\nu_0 y_{1\ttr} )- \Delta y_{2\ttr}=h_2,\\
 & \clG (y_{\ttr})\rest{\p\Omega}= 0,\qquad y_\ttr(0,\Bigcdot)= y_{\ttr0}
\end{align}
\end{subequations}
evolving in a spatial domain~$\Omega\in\bbR^d$,~$d\in\{1,2,3\}$,  a bounded convex polygonal domain. Denoting by~$(t,x)$ a generic point in the cylinder~$(0,+\infty)\times\Omega$, the state is the vector function
$y_{\ttr}=y_{\ttr}(t,x)=(y_{1\ttr}(t,x),y_{2\ttr}(t,x))\in\bbR^2$; the second component~$y_{2\ttr}$ represents the temperature and the first component~$y_{1\ttr}$ is the so-called order parameter indicating the phase transitions. The vector function~$h=h(t,x)=(h_1(t,x),h_2(t,x))\in\bbR^2$, represents the given external force, and~$y_{\ttr0}=y_{\ttr0}(x)\in\bbR^2$ is a given initial state, at time~$t=0$.

In case the external force~$h$ is independent of time, the target~$y_{\ttr}$ can be an equilibrium of the system. In this work, we consider the more general case where the target may be time-dependent. In particular, we include the case of time-dependent external forces. 

The nonlinear functions~$f\colon\bbR\to\bbR$ and~$g\colon\bbR\to\bbR$ are also given, as well as the constant vector~$\nu=(\nu_0,\nu_1,\nu_2)\in\bbR^3$ with strictly positive ~$\nu_2>0$. 
We shall consider a class of nonlinearities~$f$ including the classical double-well potential~(cf.~\cite{BarbuColliGilMar17,Caginalp88}),
\begin{align}\label{polyFtau}
F(\zeta)\coloneqq \tau(\zeta^2-1)^2,\quad\mbox{that is,}\quad f(\zeta)=F'(\zeta)=4\tau(\zeta^3-\zeta),
\end{align}
where the constant~$\tau>0$ is related with the thickness of the phase transition.
The operator~$\Delta$ stands for the usual Laplacian, formally, if~$(x_1,...,x_d)\coloneqq x$ denote the coordinates of the generic spatial point~$x\in\Omega$, 
\begin{equation}\notag
\Delta \phi \coloneqq (\tfrac{\p}{\p x_1})^2\phi+...+ (\tfrac{\p}{\p x_d})^2\phi,
\end{equation}
for a scalar function~$\phi=\phi(x)=\phi(x_1,...,x_d)$. Finally, the operator~$\clG$  represents the boundary conditions. The results cover, for a generic spatial bounded domain~$\Omega\subset\bbR^d$, both Dirichlet and Neumann boundary conditions (bcs),
\begin{subequations}\label{bcs}
\begin{align}
&\clG y_\ttr\coloneqq(y_{1\ttr}, \Delta y_{1\ttr},y_{\ttr2}),&\quad&\mbox{for Dirichlet bcs};\\
&\clG y_\ttr\coloneqq(\bfn\cdot\nabla y_{1\ttr}, \bfn\cdot \nabla \Delta  y_{1\ttr}, \bfn\cdot\nabla y_{\ttr2}),&\quad&\mbox{for Neumann bcs}.\\
\intertext{The results also cover periodic ``boundary conditions'', say}
&\clG y_\ttr \coloneqq0,&\quad&\mbox{for periodic ``bcs''},
\end{align}
where the spatial domain is the boundaryless torus~$\Omega=\bbT^d_L\sim\bigtimes\limits_{n=1}^d[0,L_n)$,~$\p\bbT^d_L=\emptyset$.
\end{subequations}

\begin{remark}
It is important, that the constant~$\nu_2>0$ is strictly positive, namely, to guarantee the existence of and uniqueness of solutions. Usually, we will also have~$\nu_0>0$ and~$\nu_1>0$; see~\cite{BarbuColliGilMar17}. Here, we consider the more general cases~$\nu_0\in\bbR$ and~$\nu_1\in\bbR$, because it allows to include the decoupled case as well, namely, the ``isothermal'' model where~$\nu_1=0$, and also because the result follows by the same arguments.
\end{remark}

%%%%%%%%%%%%%%%%%%%%%%%%%%%%%%
%%%%%%%%%%%%%%%%%%%%%%%%%%%%%%
\subsection{Stabilization to trajectories}
We assume that we are given a  reference trajectory~$y_\ttr$ of~\eqref{sys-CH} with a behavior that we would like to track. If we are given a different initial state~$y_0\ne{y_{\ttr0}}$, we may need a control to track~$y_\ttr$. 

In applications, we will likely have at our disposal a finite number of actuators only. Let~$\Phi_j=\Phi_j(x)$,~$1\le j\le M_{\sigma}$ denote the actuators acting on the order parameter equation, and let~$\Psi_j=\Psi_j(x)$,~$1\le j\le M_\varsigma$ denote the actuators acting on the heat equation.

Then, our goal is to find control inputs~$u=u(t)\in\bbR^{M_\sigma}$ and~$v=v(t)\in\bbR^{M_\varsigma}$ such that the solution~$y_\ttc$ of the controlled Cahn--Hilliard system
\begin{subequations}\label{sys-CH-control}
  \begin{align}  
 &\tfrac{\p}{\p t}y_{1\ttc} +  \nu_2\Delta^2 y_{1\ttc}  - \Delta\Bigl(f(y_{1\ttc})-\nu_1 y_{2\ttc}\Bigr)+g(y_{1\ttc})=h_1+{\textstyle\sum_{j=1}^{M_{\sigma}}}\;u_j\Phi_j,\\
  &\tfrac{\p}{\p t}(y_{2\ttc} +\nu_0 y_{1\ttc} )- \Delta y_{2\ttc}=h_2+{\textstyle\sum_{j=1}^{M_\varsigma}}\;v_j\Psi_j,\\
 & \clG (y_{\ttc})\rest{\p\Omega}= 0,\qquad y_\ttc(0,\Bigcdot)= y_{\ttc0},
\end{align}
\end{subequations}
converges to the targeted solution~$y_\ttr$ as time increases.
The scalars~$u_j=u_j(t)$ and~$v_j=v_j(t)$ stand for the coordinates of the inputs at time~$t$,~$u\eqqcolon(u_1,\dots,u_{M_\sigma})$ and~$v\eqqcolon(v_1,\dots,v_{M_\varsigma})$.

We shall show that for arbitrary given~$\mu>0$ and~$R>0$, we can find a suitable set of actuators and explicit feedback operators~$\bfK_1$ and~$\bfK_2$ such that with the inputs
  \begin{equation}\notag
u(t)=\bfK_1(y_\ttc(t,\Bigcdot)-y_\ttr(t,\Bigcdot))\in\bbR^{M_\sigma},\qquad v(t)=\bfK_2(y_\ttc(t,\Bigcdot)-y_\ttr(t,\Bigcdot))\in\bbR^{M_\varsigma},
\end{equation}
depending only on the difference between the controlled  state~$y_\ttc(t,\Bigcdot)$ and the targeted  state~$y_\ttr(t,\Bigcdot)$, at time~$t$, we will have that~$y_\ttc(t,\Bigcdot)$ converges to~$y_\ttr(t,\Bigcdot)$ exponentially, provided the initial difference has a norm smaller than~$R$, more precisely:
\begin{equation}\notag
\begin{split}
&\mbox{For every } y_{\ttc0} \mbox{ satisfying }\norm{y_{\ttc0}-y_{\ttr0}}{L^2(\Omega)}\le R\mbox{ we will have that }\\
&\quad\norm{y_\ttc(t,\Bigcdot)-y_\ttr(t,\Bigcdot)}{L^2(\Omega)\times L^2(\Omega)}\le \rme^{-\mu (t-s)}\norm{y_\ttc(s,\Bigcdot)-y_\ttr(s,\Bigcdot)}{L^2(\Omega)\times L^2(\Omega)},\\
&\quad\mbox{ for all } t\ge s\ge 0.
\end{split}
\end{equation}

The actuators and feedback are explicitly given in Sections~\ref{sS:act} and~\ref{sS:K}. Note that the rate~$\mu$ and the constant~$R>0$ are given apriori. In particular, we will be able to stabilize the system for any given initial state~$y_{\ttc0}$, provided we choose an appropriate (large enough) number of actuators depending on (an upper bound for)~$\norm{y_{\ttc0}-y_{\ttr0}}{L^2(\Omega)}$. In this sense we will have a semiglobal stabilizability result.

We shall derive the stabilizability results under general boundedness requirements on the  the external force~$h$ and on the nonlinearities~$f$ and~$g$, which are precised in Assumptions~\ref{A:h}, \ref{A:fgV}, and~\ref{A:N-dif}. We shall need a boundedness requirement for the targeted trajectory~$y_{\ttr}$ as well. This is done implicitly in Assumption~\ref{A:NfNg1} and explicitly in Assumption~\ref{A:yr-bdd} for the concrete case of the double-well potential~\eqref{polyFtau}.

To demonstrate the derived stabilizability results, we employ numerical simulations. We compare the evolution of phase-fields across varying numbers of actuators and adjust the size of the parameter associated with the feedback operator. Through these simulations, we effectively showcase our analytical findings. Specifically, we illustrate that the phase-field does not converge to its typical equilibrium state, but instead converges to its designated target. This observation underscores the predominant role of the feedback process in shaping the system dynamics.

%%%%%%%%%%%%%%%%%%%%%%%%%%%%
%%%%%%%%%%%%%%%%%%%%%%%%%%%%
\subsection{Previous literature}\label{sS:litter}
Although finite-horizon optimal control problems for various Cahn--Hilliard type systems have been studied by many authors, such as \cite{HintKei24,ColGilSigSpr23,CheBoshLiu21,GarLamKeiSig21,GarHinKah19,ColGilSpr18,HinHinKahKei18,FigRocSpr16,HintWeg12}   ,  the stabilization of these systems using feedback control has received relatively little attention. Despite the analytical interest and wide applications of the Cahn--Hilliard systems, there is very little research in this direction. One notable work is \cite{BarbuColliGilMar17}, which investigates the local stabilization of the (nonisothermal) Cahn–Hilliard equations around an equilibrium, specifically a time-independent trajectory  ~$y_\ttr(t)=y_{\ttr0}$.  In that study,  the feedback control is Riccati-based, computed by linearizing the system around the steady state and constructing the stabilizing control from a finite combination of the unstable modes of the operator in the linear system. These results were extended in \cite{marinoschi2022} for the double-well potential and in \cite{Marinoschi18} to include the viscosity effect in the phase field equation.  We can also mention \cite{Guzman20}, which investigates the local exact controllability to the trajectories of the one-dimensional Cahn--Hilliard equation.  \\
The feedback control proposed in this manuscript differs from \cite{BarbuColliGilMar17} in several  aspects. Instead of using Riccati-based feedback as in \cite{BarbuColliGilMar17}, we propose an explicit feedback based on oblique projections. This approach allows us to prove semi-global stabilizability and to stabilize around a given time-dependent trajectory. However, unlike in \cite{BarbuColliGilMar17}, where the actuator supports can be localized in an arbitrary open subset of the space domain, our feedback control requires that the actuator supports are distributed across the entire domain (still, the  total volume covered by the actuators can be taken arbitrarily small and fixed apriori).

The seminal work by Cahn and Hilliard (cf.~\cite{CahnHilliard58}) introduces a model that captures the essential features of phase separation processes through diffuse interface models, where the interface between coexisting phases is represented by a thin interfacial region, allowing partial mixing of components. Initially formulated under isothermal conditions, their model assumes isotropy and constant temperature, focusing solely on diffusive phenomena arising from instantaneous quenching below a critical temperature. 
Real-world phase separation processes often deviate from the instantaneous quench assumption, prompting investigations into nonisothermal Cahn--Hilliard models as proposed in \cite{alt1992mathematical}. These models, considering quenches over finite periods and incorporating external thermal activation for process control, offer a more realistic portrayal of phase separation dynamics. Early studies delved into the existence and uniqueness of solutions (cf.~\cite{shen1993coupled}), while analyses such as~$L^p$ maximal regularity, as explored in \cite{pruss2006maximal}, provided insight into the mathematical properties of nonisothermal Cahn--Hilliard equations.
The applicability of Cahn-Hilliard models transcends materials science, extending into diverse fields like cell biology and image processing. The book \cite{miranville2019cahn} offers comprehensive insights into Cahn--Hilliard dynamics and their applications, reflecting the broad interdisciplinary interest in this area. In particular, the reinterpretation of temperature as nutrients and the consequent analogy with tumor growth models with chemotaxis, as discussed in \cite{Fritz2023tumor}, underscores the versatility and relevance of Cahn--Hilliard frameworks in the treatment of complex biological phenomena.

%%%%%%%%%%%%%%%%%%%%%%%%%%
%%%%%%%%%%%%%%%%%%%%%%%%%%
\subsection{Contents}\label{sS:contents}
The rest of the paper is organized as follows. 
In Section~\ref{S:setting}, we gather functional spaces which are appropriate to investigate the evolution of the state of the Cahn--Hilliard system and construct the actuators and input feedback operators. The proof of the stabilization of the trajectories is given in Section~\ref{S:stab} under the appropriate assumptions. The problems concerning the existence, uniqueness and regularity of solutions are addressed in Section~\ref{S:exiuni},  under additional assumptions. The satisfiability of the required assumptions is shown in Section~\ref{S:assumok}. The results of the simulations are presented in Section~\ref{S:simul} that validates the theoretical results.
Finally, the proofs of some auxiliary results are gathered in the appendix. 

\subsection{Notation}\label{sS:notation}
Concerning notation, we write~$\bbR$ and~$\bbN$ for the sets of real numbers and nonnegative
integers, respectively, and we define~$\bbR_+\coloneqq(0\,+\infty)$,  and~$\mathbb
N_+\coloneqq\mathbb N\setminus\{0\}$.

Given Hilbert spaces~$X$ and~$Y$, if the inclusion
$X\subseteq Y$ is continuous, we write~$X\xhookrightarrow{} Y$. We write
$X\xhookrightarrow{\rm d} Y$, respectively~$X\xhookrightarrow{\rm c} Y$, if the inclusion is also dense, respectively, compact. 
The space of continuous linear mappings from~$X$ into~$Y$ is denoted by~$\clL(X,Y)$. In case~$X=Y$ we 
write~$\clL(X)\coloneqq\clL(X,X)$.
The continuous dual of~$X$ is denoted~$X'\coloneqq\clL(X,\bbR)$. The scalar product on a Hilbert space~$\clH$  is denoted~$(\Bigcdot,\Bigcdot)_\clH$. Given closed subspaces~$\clF$ and~$\clG$ of~$\clH$, in case 
$\clF\cap \clG=\{0\}$ we say that~$\clF+\clG$ is a direct sum and we write~$\clF\oplus \clG$ instead.
For a subset~$S\subseteq\clH$, its orthogonal complement is
denoted~$S^{\perp \clH}\coloneqq\{h\in \clH\mid (h,s)_\clH=0\mbox{ for all }s\in S\}$.

The space of continuous functions from a subset~$S\subseteq X$ into~$Y$ is denoted by~$\clC(S,Y)$. 

By~$\overline C_{\left[b_1,\dots,b_n\right]}$ we denote a nonnegative function that
increases in each of its nonnegative arguments~$b_j\ge0$,~$1\le j\le n$.
Finally,~$C,\,C_i$,~$i\in\bbN$, stand for unessential positive constants, which may take different values at different places in the manuscript.

%%%%%%%%%%%%%%%%%%%%%%%%%%%%
%%%%%%%%%%%%%%%%%%%%%%%%%%%%
%%%%%%%%%%%%%%%%%%%%%%%%%%%%
\section{Functional setting and families of actuators}\label{S:setting}
We introduce the appropriate subspaces involved in the analysis of the Cahn--Hilliard equations, as well as a strategy to construct explicit and appropriate families~$U_M$ and~$V_M$ of actuators and families~$\widetilde U_M$ and~$\widetilde V_M$  of auxiliary  functions.   

%%%%%%%%%%%%%%%%%%%%%%
\subsection{Function spaces}\label{sS:fun-sett}
The Lebesgue space~$H\coloneqq L^2(\Omega)$ is considered as the pivot space,~$H=H'$.
Next, we also introduce the spaces associated
to the shifted Laplacian~$A\colon V\to V'$, under one of the boundary conditions in~\eqref{bcs}. Namely, we introduce the Hilbert space
 \begin{subequations}\notag
 \begin{align}
 &V\coloneqq \begin{cases}W^{1,2}_0(\Omega)=\{\phi\in W^{1,2}(\Omega)\mid  y_{\ttr}\rest{\p\Omega}=0\},&\quad\mbox{for Dirichlet bcs};\\
 W^{1,2}(\Omega),&\quad\mbox{for Neumann or periodic bcs};
 \end{cases}
 \intertext{and the shifted Laplacian is given by}
&\langle A w,z\rangle_{V',V}\coloneqq(\nabla w,\nabla z)_{L^2(\Omega)^d}+(w,z)_{L^2(\Omega)}.
  \end{align}
\end{subequations}
with domain
 \begin{equation}\notag
\rmD(A) \coloneqq\{\phi\in H\mid A \phi\in H\}= \begin{cases}\{\phi\in W^{2,2}(\Omega)\mid  \phi\rest{\p\Omega}=0\},&\quad\mbox{for Dirichlet bcs};\\
\{\phi\in W^{2,2}(\Omega)\mid  \bfn\cdot\nabla \phi\rest{\p\Omega}=0\},&\quad\mbox{for Neumann bcs};\\
W^{2,2}(\Omega),&\quad\mbox{for periodic bcs}.
\end{cases}
\end{equation}
Hereafter, the spaces above are assumed endowed with the scalar products as follows,
 \begin{equation}\notag
(w,z)_H \coloneqq (w,z)_{L^2(\Omega)};\qquad(w,z)_V \coloneqq \langle A w,z\rangle_{V',V};\qquad(w,z)_{\rmD(A)} \coloneqq (Aw,Az)_H.
\end{equation}
In particular we see that~$A\colon V\to V'$ and~$A\colon \rmD(A)\to H$ are isometries. It will be convenient to introduce the (nonshifted) Laplacian operator~$A_0$ and seminorm~$\dnorm{\Bigcdot}{V_0}$ as follows
 \begin{subequations}\notag
 \begin{align}
 A_0\coloneqq A-\Id;\qquad&((w,z))_{V_0} \coloneqq\langle A_0 w,z\rangle_{V',V} \coloneqq(\nabla w,\nabla z)_{L^2(\Omega)^d},\\
&\dnorm{z}{V_0}^2 \coloneqq ((z,z))_{V_0}.     
\end{align}
\end{subequations}

%%%%%%%%%%%%%%%%%%%%%%%%%%%%
%%%%%%%%%%%%%%%%%%%%%%%%%%%%
\subsection{The actuators}\label{sS:act}
For a given rectangular domain~$\Omega=\bigtimes_{n=1}^d(0,L_n)$,
we fix a small domain~$\omega\subset\bbR^d$ with center of mass at the origin and we assume that
~$\omega$ is small enough for us to include~$d+2$ disjoint translations of~$\omega$ in~$\Omega$.
For simplicity, we shall take~$\omega$ as a rectangular domain
\begin{align}\label{omega}
\omega\coloneqq\bigtimes_{n=1}^d(-l_n,l_n).
\end{align}
Furthermore, we assume that there is a subset of~$d+1$ of those~$d+2$ translations of~$\omega$ contained in no affine hyperplane. We denote these by~$\omega_{1j}^1$,~$1\le j\le d+1$, and the remaining by~$\omega_{21}^1$.
We construct the sequence of pairs of families of actuators as follows.
In the case~$M=1$ we take indicator functions of the translated domains above
\begin{align}
U_1\coloneqq\{\Phi_j^1=\indf_{\omega_{1j}^1}\mid 1\le j\le d+1\}\subset H,\qquad V_1\coloneqq\{\Psi_j^1=\indf_{\omega_{2j}^1}\mid  j=1\}\subset H.\notag
\end{align}
In the case~$M>1$ we divide the rectangular domain into~$M^{d}$ rescaled copies (up to a translation) of itself including the  domains of the actuators, as illustrated in Fig.~\ref{fig.suppActSqu}.   This leads us to families as
\begin{align}\label{UVM}
U_M\coloneqq\{\Phi_j^M=\indf_{\omega_{1j}^M}\mid 1\le j\le M_\sigma\}\subset H,\qquad V_M\coloneqq\{\Psi_j^M=\indf_{\omega_{2j}^M}\mid 1\le j\le M_\varsigma\}\subset H,
\end{align}
with~$(M_\sigma,M_\varsigma)=((d+1)M^{d},1M^{d})$.

%%%%%%%%%%%%%%%%%%%%%%%%TIKZ BOXES
%%%% RECT %%%%%%%%%
\setlength{\unitlength}{.002\textwidth}
\newsavebox{\Rectfw}
\savebox{\Rectfw}(0,0){%
\linethickness{3pt}
{\color{black}\polygon(0,0)(120,0)(120,80)(0,80)(0,0)}%
}
\newsavebox{\RectrefRG}
\savebox{\RectrefRG}(0,0){%
{\color{white}\polygon*(0,0)(120,0)(120,80)(0,80)(0,0)}%
{\color{red}\polygon*(22.5,15)(37.5,15)(37.5,25)(22.5,25)(22.5,15)}%
{\color{red}\polygon*(22.5,55)(37.5,55)(37.5,65)(22.5,65)(22.5,55)}%
{\color{red}\polygon*(82.5,15)(97.5,15)(97.5,25)(82.5,25)(82.5,15)}%
{\color{green}\polygon*(82.5,55)(97.5,55)(97.5,65)(82.5,65)(82.5,55)}%
}
%%%%%%%%%%%%%%%%%%%%%%%%%%

\begin{figure}[ht!]
\begin{center}
\begin{picture}(500,100)%(0,0)
% % Rect2x2
\put(0,0){\usebox{\Rectfw}}
\put(0,0){\usebox{\RectrefRG}}
% % Rect4x4
\put(190,0){\usebox{\Rectfw}}
\put(190,0){\scalebox{.5}{\usebox{\RectrefRG}}}
 \put(250,0){\scalebox{.5}{\usebox{\RectrefRG}}}
 \put(250,40){\scalebox{.5}{\usebox{\RectrefRG}}}
 \put(190,40){\scalebox{.5}{\usebox{\RectrefRG}}}
% % Rect6x6
\put(380,0){\usebox{\Rectfw}}
\put(380,0){\scalebox{.333}{\usebox{\RectrefRG}}}
\put(380,26.666){\scalebox{.333}{\usebox{\RectrefRG}}}
\put(380,52.333){\scalebox{.333}{\usebox{\RectrefRG}}}
\put(420,0){\scalebox{.333}{\usebox{\RectrefRG}}}
\put(420,26.666){\scalebox{.333}{\usebox{\RectrefRG}}}
\put(420,52.333){\scalebox{.333}{\usebox{\RectrefRG}}}
\put(460,0){\scalebox{.333}{\usebox{\RectrefRG}}}
\put(460,26.666){\scalebox{.333}{\usebox{\RectrefRG}}}
\put(460,52.333){\scalebox{.333}{\usebox{\RectrefRG}}}
\put(40,85){$M=1$}
\put(230,85){$M=2$}
\put(420,85){$M=3$}
\put(14,20){$\omega_{11}^1$}
\put(14,62){$\omega_{13}^1$}
\put(74,20){$\omega_{12}^1$}
\put(74,62){$\omega_{21}^1$}
\linethickness{2pt}
{\color{blue}
\Dashline(460,0)(500,0){2}
\Dashline(500,0)(500,26.6666){2}
\Dashline(500,26.6666)(460,26.6666){2}
\Dashline(460,26.6666)(460,0){2}
}
\end{picture}
\end{center}
\caption{Supports of actuators, for rectangular~$\Omega\subset\bbR^2$.~$(M_\sigma,M_\varsigma)=M^{2}(3,1)$.} \label{fig.suppActSqu}
\end{figure}
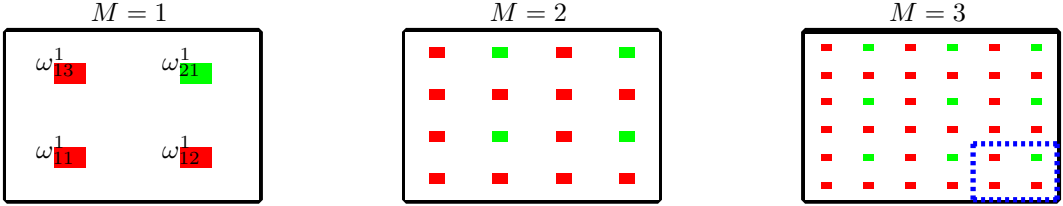

In the case~$d=2$ we can follow an analogue procedure for a triangular domain taking rescaled copies (up to a translation and a rotation) as illustrated in Fig.~\ref{fig.suppActTri}.

%%%%%%%%%%%%%%%%%%%%%%%%TIKZ BOXES
%%%% TRI %%%%%%%%%
\setlength{\unitlength}{.002\textwidth}
\newsavebox{\Trifw}
\savebox{\Trifw}(0,0){%
\linethickness{3pt}
{\color{black}\polygon(0,0)(120,0)(40,80)(0,0)}%
}
\newsavebox{\TrirefRG}
\savebox{\TrirefRG}(0,0){%
{\color{white}\polygon*(0,0)(120,0)(40,80)(0,0)}%
{\color{red}\polygon*(20,10)(35,10)(35,20)(20,20)(20,10)}%
{\color{red}\polygon*(80,10)(95,10)(95,20)(80,20)(80,10)}%
{\color{red}\polygon*(40,50)(55,50)(55,60)(40,60)(40,50)}%
{\color{green}\polygon*(45,20)(60,20)(60,30)(45,30)(45,20)}%
}
\newsavebox{\rTrirefRG}
\savebox{\rTrirefRG}(0,0){\rotatebox{180}{\usebox{\TrirefRG}}%
}
%%%%%%%%%%%%%%%%%%%%%%%%%%%%%%%%%%%

\begin{figure}[ht!]
\begin{center}
\begin{picture}(300,100)%(0,0)
%Tri1
\put(0,0){\usebox{\Trifw}}
\put(0,0){\usebox{\TrirefRG}}
% Tri2
 \put(190,0){\usebox{\Trifw}}
 \put(190,0){\scalebox{.5}{\usebox{\TrirefRG}}}
  \put(250,0){\scalebox{.5}{\usebox{\TrirefRG}}}
 \put(210,40){\scalebox{.5}{\usebox{\TrirefRG}}}
 \put(270,40){\scalebox{.5}{\usebox{\rTrirefRG}}}
\put(30,85){$M=1$}
\put(220,85){$M=2$}
\put(12,15){$\omega_{11}^1$}
\put(33,55){$\omega_{13}^1$}
\put(72,15){$\omega_{12}^1$}
\put(35,27){$\omega_{21}^1$}
\linethickness{1.5pt}
{\color{blue}%
\Dashline(270,40)(210,40){4}
\Dashline(210,40)(250,0){4}
\Dashline(250,0)(270,40){4}
}
\end{picture}
\end{center}
\caption{Supports of  actuators, for  triangular~$\Omega\subset\bbR^2$.~$(M_\sigma,M_\varsigma)=4^{M-1}(3,1)$. }
\label{fig.suppActTri}
\end{figure}
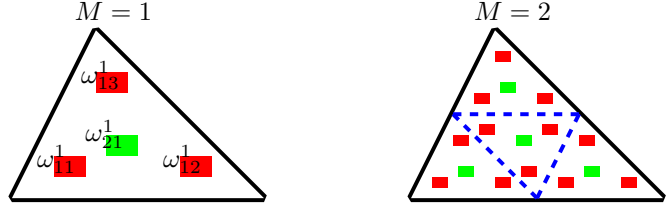

In addition, we shall need a set of auxiliary functions.
First of all, note that we have~$\indf_{\omega_{ij}^M}(x)=\indf_{\omega}(\fkR_{ij}^M(\fkH^M(\fkT_{ij}^M (x))))$ for a composition of a translation~$\fkT_{ij}^M$, a homotethy~$\fkH^M$, and a rotation~$\fkR_{ij}^M$. Namely
\begin{align}
\fkT_{ij}^M (x)\coloneqq x-c_{ij},\qquad\fkH^M z\coloneqq M^{-1}z,\qquad \fkR_{ij}^Mw=w,\notag
\end{align}
for rectangular domains where no nontrivial rotation is needed and, for the case of triangular planar domains, with rotations as
\begin{align}
\fkR_{ij}^Mw=\begin{bmatrix}\cos(\theta)&-\sin(\theta)\\\sin(\theta)&\cos(\theta)\end{bmatrix}w, \qquad\theta\in[0,2\pi).\notag
\end{align}
For example, in Fig.~2 for the case~$M=2$, we used thrice the (trivial) rotation with angle~$\theta=0$ and once the rotation with angle~$\theta=\pi$. We construct the auxiliary functions as follows. We introduce the function
\begin{subequations}\notag
\begin{align}
\varphi(z)\coloneqq \indf_\omega(z)\bigtimes_{n=1}^d\sin(\pi\tfrac{z_n+l_n}{2l_n}),\quad z\in\bbR^d,
\end{align}
with~$\omega$ as in~\eqref{omega}; note that~$\varphi$ is piecewise smooth and vanishes at the boundary of~$\omega$. Then, we use take
\begin{align}
\widetilde \Phi_j^M(x)\coloneqq \varphi^2(\fkR_{ij}^M(\fkH^M(\fkT_{ij}^M (x)))),\qquad \widetilde \Psi_j^M(x)\coloneqq \varphi(\fkR_{ij}^M(\fkH^M(\fkT_{ij}^M (x)))),
\end{align}
to build the sets of auxiliary functions as
\begin{align}
\widetilde U_M\coloneqq\{\widetilde \Phi_j^M\mid 1\le j\le M_\sigma\}\subset\rmD(A),\qquad \widetilde V_M\coloneqq\{\widetilde \Psi_j^M\mid 1\le j\le M_\varsigma\}\subset V.
\end{align}
\end{subequations}
Hence, the taken auxiliary functions can be seen as ``regularized actuators''. By introducing the linear spans of actuators and auxiliary functions as
\begin{equation}\notag
\clU_M\coloneqq\linspan U_M;\quad \clV_M\coloneqq\linspan V_M,\quad \widetilde \clU_M\coloneqq\linspan\widetilde U_M;\quad \widetilde \clV_M\coloneqq\linspan\widetilde V_M,
\end{equation}
we are ready to present two key auxiliary results as follows.
\begin{lemma}\label{L:poincare}
We have that the sequences~$(\alpha_{M}^H)_{M\in\bbN_+}$ and~$(\alpha_{M}^V)_{M\in\bbN_+}$ of constants
\begin{align}
&\alpha_{M}^H\coloneqq\inf_{w\in(V\bigcap\clV_M^\perp)\setminus\{0\}}\frac{\norm{w}{V}^2}{\norm{w}{H}^2};\qquad&&\alpha_{M}^{V}\coloneqq\inf_{w\in(\rmD(A)\bigcap\clU_M^\perp)\setminus\{0\}}\frac{\norm{w}{\rmD(A)}^2}{\norm{w}{V}^2};\notag
\end{align}
are divergent;~$\lim_{M\to+\infty}\alpha_{M}^H=+\infty$ and~$\lim_{M\to+\infty}\alpha_{M}^{V}=+\infty$.
 \end{lemma}
\begin{lemma}\label{L:M-and-lam}
Given~$\xi>0$, we can find~$\overline M\in\bbN_+$,~$\overline\lambda_1\ge0$, and~$\overline\lambda_2\ge0$ large enough so that
\begin{align}
&\norm{y}{\rmD(A)}^2+2\lambda_1 \norm{P_{\widetilde\clU_{M}}^{\clU_{M}^\perp}y}{\rmD(A)}^2
\ge \xi \norm{y}{V}^2, &&\quad\mbox{for all}\quad
y\in \rmD(A),\quad M\ge\overline M,\quad \lambda_1\ge\overline\lambda_1;\label{L:M-and-lamV}\\
&\norm{y}{V}^2+2\lambda_2 \norm{P_{\widetilde\clV_{M}}^{\clV_{M}^\perp}y}{V}^2
\ge  \xi  \norm{y}{H}^2, &&\quad\mbox{for all}\quad
y\in V,\quad M\ge\overline M,\quad \lambda_2\ge\overline\lambda_2.\label{L:M-and-lamH}
\end{align}
Furthermore~$\overline M=\ovlineC{\xi}$,~$\overline\lambda_1=\ovlineC{\xi}$, and~$\overline\lambda_2=\ovlineC{\xi}$.
 \end{lemma}

The proofs of Lemmas~\ref{L:poincare} and~\ref{L:M-and-lam} follow by slight variations of arguments in the literature, we present details in Appendix~\ref{Apx:proofL:poincare} and~\ref{Apx:proofL:M-and-lam}.

%%%%%%%%%%%%%%%%%%%%%%%%%%%%
%%%%%%%%%%%%%%%%%%%%%%%%%%%%
\subsection{Feedback}\label{sS:K} We shall denote
  \begin{subequations}\label{OP-FeedK}
\begin{equation}
(y_1,y_2)\coloneqq y\coloneqq y_{\ttc}-y_{\ttr}.
\end{equation}
With the actuators as in~\eqref{UVM} we introduce the isomorphisms
\begin{align}\notag
U_M^\diamond u\coloneqq {\textstyle\sum_{j=1}^{M_{\sigma}}}\;u_j\Phi_j^M,\qquad
V_M^\diamond v\coloneqq {\textstyle\sum_{j=1}^{M_\varsigma}}\;v_j\Psi_j^M,
\end{align}
and take the explicit feedback with components as follows 
 \begin{align}
&\bfK_1y\coloneqq -\lambda_1(U_M^\diamond)^{-1}P_{\clU_{M}}^{\widetilde\clU_{M}^\perp}A^2P_{\widetilde\clU_{M}}^{\clU_{M}^\perp}y_1,&&\mbox{with}\quad\lambda_1\ge0;\\
&\bfK_2 y\coloneqq-\lambda_2(V_M^\diamond)^{-1}P_{\clV_{M}}^{\widetilde\clV_{M}^\perp} AP_{\widetilde\clV_{M}}^{\clV_{M}^\perp} (y_{2} +\nu_0 y_{1}),&&\mbox{with}\quad\lambda_2\ge0;
\end{align}
\end{subequations}
where constants~$M$,~$\lambda_1$, and~$\lambda_2$, shall be chosen large enough in order to guarantee that~$y=y_{\ttc}-y_{\ttr}$ converges to zero  exponentially as time increases.  With the feedback input components as above, the control dynamics~\eqref{sys-CH-control} reads
\begin{subequations}\label{sys-y-CH-Feed}
  \begin{align}  
 &\tfrac{\p}{\p t}y_{1\ttc} +  \nu_2\Delta^2 y_{1\ttc}  - \Delta\Bigl(f(y_{1\ttc})-\nu_1 y_{2\ttc}\Bigr)+g(y_{1\ttc})=h_1+U_M^\diamond\bfK_1y,\\
  &\tfrac{\p}{\p t}(y_{2\ttc} +\nu_0 y_{1\ttc} )- \Delta y_{2\ttc}=h_2+V_M^\diamond\bfK_2y,\\
 & \clG (y_{\ttc})\rest{\p\Omega}= 0,\qquad y_\ttc(0,\Bigcdot)= y_{\ttc0}.
\end{align}
\end{subequations}

 \begin{subequations}\label{Kmonotone}
Since~$P_{\widetilde\clU_{M}}^{\clU_{M}^\perp}=(P_{\clU_{M}}^{\widetilde\clU_{M}^\perp})^*$ (cf.~\cite[Lem.~3.4]{KunRodWal21} \cite[Lem.~3.8]{RodSturm20}),  the operators
 \begin{align}
&\bfF_1w\coloneqq U_M^\diamond\bfK_1w=-\lambda_1P_{\clU_{M}}^{\widetilde\clU_{M}^\perp}A^2P_{\widetilde\clU_{M}}^{\clU_{M}^\perp}w,\qquad \bfF_2 w\coloneqq V_M^\diamond\bfK_2w=-\lambda_2P_{\clV_{M}}^{\widetilde\clV_{M}^\perp} AP_{\widetilde\clV_{M}}^{\clV_{M}^\perp} w,
\end{align}
satisfy the  relations
\begin{equation}
(\bfF_1w,w)_H=-\lambda_1\norm{P_{\widetilde\clU_{M}}^{\clU_{M}^\perp}w}{\rmD(A)}^2\quad\mbox{and}\quad (\bfF_2w,w)_H=-\lambda_2\norm{P_{\widetilde\clV_{M}}^{\clV_{M}^\perp}w}{V}^2.
\end{equation}
\end{subequations}

%%%%%%%%%%%%%%%%%%%%%%%%%%%
%%%%%%%%%%%%%%%%%%%%%%%%%%%
%%%%%%%%%%%%%%%%%%%%%%%%%%%
\section{Stability of the difference to the target}\label{S:stab}
We show the stability of the dynamics of the difference~$y\coloneqq y_{\ttc}-y_{\ttr}$ to the target~$y_{\ttr}$, for~$M$,~$\lambda_1$, and~$\lambda_2$ large enough, under general assumptions on the nonlinearities. For~$y=(y_1,y_2)=y_{\ttc}-y_{\ttr}$, we find
\begin{subequations}\notag
  \begin{align}  
 &\tfrac{\p}{\p t}y_{1} +  \nu_2\Delta^2 y_{1}  - \Delta\Bigl(f(y_{1\ttc})-f(y_{1\ttr})-\nu_1 y_{2}\Bigr)+g(y_{1\ttc})-g(y_{1\ttr})=U_M^\diamond \bfK_1 y,\\
  &\tfrac{\p}{\p t}(y_{2} +\nu_0 y_{1} )- \Delta y_{2}=V_M^\diamond \bfK_2 y,\\
 & \clG y\rest{\p\Omega}= 0,\qquad y(0,\Bigcdot)= y_{0}= y_{\ttc0}-y_{\ttr0}.
\end{align}
\end{subequations}
By introducing the change of variables
\begin{equation}\label{chvar-y-to-z}
z=(z_1,z_2)\coloneqq\Xi y\coloneqq (y_1,y_{2} +\nu_0 y_{1}),
\end{equation}
and recalling~\eqref{OP-FeedK} and~\eqref{Kmonotone},  we obtain
\begin{subequations}\label{sys-z}
  \begin{align}  
 &\tfrac{\p}{\p t}z_{1} =-  \nu_2\Delta^2 z_{1}  -\nu_1\Delta(z_{2}-\nu_0z_1)+\Delta N_f(z_1)-N_g(z_1)+\bfF_1 z_1,\\
  &\tfrac{\p}{\p t}z_{2}= \Delta z_{2}- \nu_0\Delta z_{1}+  \bfF_2z_2,\\
 & \clG z\rest{\p\Omega}= 0,\qquad z(0,\Bigcdot)= z_{0}\coloneqq\Xi y_{0},
\intertext{with}
 &N_f(z_1)=N_f(t,z_1)\coloneqq f(y_{1\ttr}(t)+z_1)-f(y_{1\ttr}(t)),\\
 &N_g(z_1)=N_g(t,z_1)\coloneqq g(y_{1\ttr}(t)+z_1)-g(y_{1\ttr}(t)).
\end{align}
\end{subequations}

\begin{assumption}\label{A:NfNg1}
There exists an integer~$n\in\bbN_+$ and nonnegative constants~$C_f$, $C_g$,$\zeta_i^f$, $\eta_i^f$,$\zeta_i^g$,$\eta_i^g$, $1\le i\le n$, such that for almost all~$t>0$
\begin{align}
-\langle A_0 N_f(t,z_1),z_1\rangle_{\rmD(A)',\rmD(A)}&\le C_f\sum_{i=1}^n\norm{z_1}{H}^{\zeta_i^f}\norm{z_1}{\rmD(A)}^{\eta_i^f};\notag\\
-\langle N_g(t,z_1),z_1\rangle_{\rmD(A)',\rmD(A)}&\le C_g\sum_{i=1}^n\norm{z_1}{H}^{\zeta_i^g}\norm{z_1}{\rmD(A)}^{\eta_i^g};\notag
\end{align}
with~$\eta_i^f<2$,  $\eta_i^g<2$, and~$\zeta_i^f+\eta_i^f\ge2\le\zeta_i^g+\eta_i^g$.
\end{assumption}

The satisfiability of Assumption~\ref{A:NfNg1} will, in particular, require suitable boundedness properties for the target reference trajectory~$y_{1\ttr}$. We shall illustrate this point in Section~\ref{S:assumok}, where we  will see (cf. Assum.~\ref{A:yr-bdd}) that,    it holds for the nonlinearity~$f$ associated with the double-well potential~\eqref{polyFtau}, a suitable property is that~$y_{1\ttr}\in L^{\infty}(\bbR_+\times\Omega)$ and~$\nabla y_{1\ttr}\in L^\infty(\bbR_+\times\Omega)^d$. For example, in the autonomous case and with vanishing external forcing~$h$, there are equilibria~$\varphi\in W^{4,2}(\Omega)$ as shown in~\cite[Lem.~4.1]{BarbuColliGilMar17}. In particular,  for these equilibria trajectories we have that~$y_{1\ttr}(t)=\varphi\in L^{\infty}(\Omega)$ and~$\nabla y_{1\ttr}(t)=\nabla \varphi\in L^\infty(\Omega)^d$, due to the Agmon inequality, recall that~$d\in\{1,2,3\}$. See also Remark~\ref{R:equil} for the case of nonvanishing external forcings.

\begin{corollary}\label{C:NfNg1}
Let Assumption~\ref{A:NfNg1} hold true. Then, for any given~$\gamma>0$, there are nonnegative constants~$D_f$, $D_g$, $\zeta^f$, $\zeta^g$ such that for almost all~$t>0$, 
\begin{align}
-\langle A_0 N_f(t,z_1),z_1\rangle_{\rmD(A)',\rmD(A)}&\le D_f(1+\norm{z_1}{H}^{\zeta^f})\norm{z_1}{H}^2+\gamma\norm{z_1}{\rmD(A)}^{2};\notag\\
-\langle N_g(t,z_1),z_1\rangle_{\rmD(A)',\rmD(A)}&\le D_g(1+\norm{z_1}{H}^{\zeta^g})\norm{z_1}{H}^2+\gamma\norm{z_1}{\rmD(A)}^{2};\notag
\end{align}
with~$D_f=\ovlineC{C_f,\gamma^{-1}}$ and~$D_g=\ovlineC{C_g,\gamma^{-1}}$.
\end{corollary}

\begin{proof}
The result follows by the Young inequality. In the case~$\eta_i^f\in(0,1)$, for each~$\gamma>0$, taking~$\varepsilon\coloneqq (nC_f)^{-1}\gamma$,
\begin{align}
\norm{z_1}{H}^{\zeta_i^f}\norm{z_1}{\rmD(A)}^{\eta_i^f}=\varepsilon^{-\frac{\eta_i^f}{2}}\norm{z_1}{H}^{\zeta_i^f}\varepsilon^{\frac{\eta_i^f}{2}}\norm{z_1}{\rmD(A)}^{\eta_i^f}\le \varepsilon^{-\frac{\eta_i^f}{2-\eta_i^f}}\norm{z_1}{H}^{\frac{2\zeta_i^f}{2-\eta_i^f}}+\varepsilon\norm{z_1}{\rmD(A)}^{2}\notag
\end{align}
and we see that~$\frac{2\zeta_i^f}{2-\eta_i^f}\ge\frac{2(2-\eta_i^f)}{2-\eta_i^f}=2$. Thus, with~$\beta_i^f\coloneqq 2-\frac{2\zeta_i^f}{2-\eta_i^f}\ge0$, we obtain
\begin{align}
\norm{z_1}{H}^{\zeta_i^f}\norm{z_1}{\rmD(A)}^{\eta_i^f}\le \varepsilon^{-\frac{\eta_i^f}{2-\eta_i^f}}\norm{z_1}{H}^{\beta_i^f}\norm{z_1}{H}^2+\varepsilon\norm{z_1}{\rmD(A)}^{2}.\notag
\end{align}

Next in the case~$\eta_i^f=0$, with~$\beta_i^f=2-\zeta_i^f\ge0$, it holds
\begin{align}
\norm{z_1}{H}^{\zeta_i^f}\norm{z_1}{\rmD(A)}^{\eta_i^f}=\norm{z_1}{H}^{\zeta_i^f}=\norm{z_1}{H}^{\beta_i^f}\norm{z_1}{H}^2\le \norm{z_1}{H}^{\beta_i^f}\norm{z_1}{H}^2+\varepsilon\norm{z_1}{\rmD(A)}^{2}.\notag
\end{align}
Therefore, with~$C_0=\max\Bigl\{1,\max\limits_{1\le i\le n}\Bigl\{\varepsilon^{-\frac{\eta_i^f}{2-\eta_i^f}}\Bigr\}\Bigr\}$ we arrive at
\begin{align}
-\langle A_0 N_f(t,z_1),z_1\rangle_{\rmD(A)',\rmD(A)}&\le C_f C_0\sum_{i=1}^n\norm{z_1}{H}^{\beta_i^f}\norm{z_1}{H}^2+nD_f\varepsilon\norm{z_1}{\rmD(A)}^{2}\notag\\
&\le D_f(1+\norm{z_1}{H}^{\zeta^f})\norm{z_1}{H}^2+\gamma\norm{z_1}{\rmD(A)}^{2}\notag
\end{align}
with~$D_f\coloneqq nC_f C_0$ and~$\zeta^f=\max\limits_{1\le i\le n}\{\beta_i^f\}$.
Finally, note that we can finish the proof by repeating the same argument for the terms~$N_g(t,z_1)$ and~$\norm{z_1}{H}^{\zeta_i^g}\norm{z_1}{\rmD(A)}^{\eta_i^g}$.
\end{proof}

To simplify the exposition, we write the dynamics of~$z$ as 
  \begin{align}  
 \dot z_{1} &=  -\nu_2A_0^2 z_{1}-\nu_1\nu_0A_0 z_{1}+ \nu_1A_0z_{2}-A_0 N_f(z_1)-N_g(z_1)+ \bfF_1z_1\notag\\
   \dot z_{2}&= -A_0 z_{2} +\nu_0A_0z_{1}+  \bfF_2z_2,\notag
\end{align}
which we can also rewrite in matrix form as the evolutionary equation
\begin{subequations}\label{sys-z-evol}
\begin{align}
\dot z&=-\bfA z-\bfA_{\rm rc} z-\bfN(z)+\bfF z,\qquad z(0)= z_{0}
\intertext{with the linear operators}
\bfA&\coloneqq\begin{bmatrix}
 \nu_2A^2&\zero\\
\zero&A
\end{bmatrix};\qquad \bfA_{\rm rc}\coloneqq\begin{bmatrix}(\nu_0\nu_1-2\nu_2) A_0-\nu_2\Id&-\nu_1A_0\\-\nu_0A_0&-\Id
\end{bmatrix};
\intertext{and the nonlinear and control operators as}
\bfN(z)&\coloneqq\begin{bmatrix}
A_0N_f(z_1)+N_g(z_1)\\
\zero
\end{bmatrix};\qquad \qquad\bfF\coloneqq\begin{bmatrix}\bfF_1&\zero\\\zero&\bfF_2
\end{bmatrix}.
\end{align}
\end{subequations}

\subsection{Main result}
The stabilizability to a given reference trajectory shall follow as a corollary of the following result.
\begin{theorem}\label{T:main} Let Assumption~\ref{A:NfNg1} hold true and let us take the actuators and auxiliary functions as in Section~\ref{sS:act}. Then, for any given~$\mu>0$ and~$R>0$, there are~$\overline M\in\bbN_+$,  $\overline \lambda_1\ge0$ and~$\overline \lambda_2\ge0$
such that, if~$M\ge\overline  M$, $\lambda_1\ge\overline \lambda_1$, and~$\lambda_2\ge\overline \lambda_2$, then the solution of~\eqref{sys-z-evol} satisfies
\begin{align}\label{exp-mu-z}
\norm{z(t)}{H\times H}^2&\le\rme^{-2\mu(t-s)}\norm{z(s)}{H\times H}^2,\quad\mbox{for every }z_{0}\mbox{ such that }\norm{z_0}{H\times H}\le R. 
\end{align}
Moreover,  for all~$t\ge s\ge0$, denoting~$\bbR_{s+}\coloneqq(s,+\infty)$,
\begin{align}\label{energy-z}
&\norm{z(t)}{H\times H}^2+\tfrac{\nu_2}2\norm{z_1}{L^2( \bbR_{s+},\rmD(A))}^2+\tfrac12\norm{z_2}{L^2(\bbR_{s+},V)}^2+2\mu\norm{z}{L^2(\bbR_{s+},H\times H)}^2\le\norm{z(s)}{H\times H}^2.
\end{align}
Furthermore, $\overline M=\ovlineC{\mu,C_f,C_g,\nu_0,\nu_1,\nu_2^{-1}}$,  $\overline \lambda_1=\ovlineC{\mu,C_f,C_g,\nu_0,\nu_1,\nu_2^{-1}}$, and~$\overline \lambda_2=\ovlineC{\mu}$
\end{theorem}

%%%%%%%%%%%%%%%%%%%%%%
%%%%%%%%%%%%%%%%%%%%%%
\begin{proof}
By multiplying the dynamics, in~\eqref{sys-z-evol}, by~$2z=(2z_1,2z_2)$, we obtain
\begin{align}
\tfrac{\rmd}{\rmd t}\norm{z}{H\times H}^2
&=-2\nu_2\norm{z_1}{\rmD(A)}^2-2\norm{z_2}{V}^2+2(2\nu_2-\nu_0\nu_1)\dnorm{z_1}{V_0}^2+2\nu_2\norm{z_1}{H}^2+2\norm{z_2}{H}^2\notag\\
&\quad+2(\nu_0+\nu_1)(A_0z_1,z_2)_H-2\langle A_0N_f(z_1),z_1\rangle_{\rmD(A)',\rmD(A)} -2\langle N_g(z_1),z_1\rangle_{\rmD(A)',\rmD(A)}\notag\\
&\quad+2(\bfF z,z)_{H\times H}.\notag
\end{align}
Now, using Corollary~\ref{C:NfNg1}, with~$\gamma=\nu_2>0$, $D=\max\{D^f,D^g\}\ge0$, $\zeta=\max\{\zeta^f,\zeta^g\}\ge0$, in combination with the Young inequality
\begin{align}
\tfrac{\rmd}{\rmd t}\norm{z}{H\times H}^2&\le
-2\nu_2\norm{z_1}{\rmD(A)}^2+(4\nu_2+(\nu_0+\nu_1)^2)\norm{z_1}{V}^2 +2D(1+\norm{z_1}{H}^\zeta)\norm{z_1}{H}^2+\nu_2\norm{z_1}{\rmD(A)}^2\notag\\
&\quad-\norm{z_2}{V}^2+\norm{z_2}{H}^2+2( \bfF z,z)_{H\times H}\notag
\end{align}
with~$D=\ovlineC{C_f,C_g,\nu_2^{-1}}\ge0$, where~$C_f$ and~$C_g$ are as Assumption~\ref{C:NfNg1}. Next, from~\eqref{Kmonotone},  we can infer that
\begin{align}
\tfrac{\rmd}{\rmd t}\norm{z}{H\times H}^2&\le-\nu_2\norm{z_1}{\rmD(A)}^2+(4\nu_2+(\nu_0+\nu_1)^2+2D+2D\norm{z_1}{H}^\zeta)\norm{z_1}{H}^2\notag\\
&\quad-2\lambda_1\norm{P_{\widetilde\clU_{M}}^{\clU_{M}^\perp}z_1}{\rmD(A)}^2-\norm{z_2}{V}^2+\norm{z_2}{H}^2-2\lambda_2\norm{P_{\widetilde\clV_{M}}^{\clV_{M}^\perp}z_2}{V}^2.\notag
\end{align}
Using Lemma~\ref{L:M-and-lam},    namely  \eqref{L:M-and-lamV}  for  $z_1$ with 
 \[\xi := 2\nu_2^{-1}\left( 4\nu_2+(\nu_0+\nu_1)^2+2D +2\mu+2DR^\zeta \right), \]  and  \eqref{L:M-and-lamH}   for $z_2$ with $\xi := 2\left(1+2\mu \right)$,  we can conclude that  there exists~$(\overline M,\overline \lambda_1,\overline \lambda_2)\in\bbN_+\times\overline\bbR_+\times\overline\bbR_+$ as
\begin{align}
\overline M=\ovlineC{\mu,C_f,C_g,\nu_0,\nu_1,\nu_2,\nu_2^{-1},R},\quad \overline\lambda_1=\ovlineC{\mu,C_f,C_g,\nu_0,\nu_1,\nu_2,\nu_2^{-1},R},  \quad \mbox{and}\quad \overline\lambda_2=\ovlineC{\mu},\notag
\end{align}
such that for all~$M\ge\overline M$, $\lambda_1\ge \overline\lambda_1$, and~$\lambda_2\ge \overline\lambda_2=\ovlineC{\mu}$, 
we have
\begin{align}
\tfrac{\rmd}{\rmd t}\norm{z}{H\times H}^2&\le-\frac{\nu_2}{2}\norm{z_1}{\rmD(A)}^2-\frac12\norm{z_2}{V}^2-(2\mu+2DR^\zeta-2D\norm{z_1}{H}^\zeta)\norm{z_1}{V}^2-2\mu\norm{z_2}{H}^2\notag\\
&= -\frac{\nu_2}{2}\norm{z_1}{\rmD(A)}^2-\frac12\norm{z_2}{V}^2-2\mu\norm{(z_1,z_2)}{H\times H}^2
-2D(R^\zeta-\norm{z_1}{H}^\zeta)\norm{z_1}{H}^2\label{energy-z-2}\\
&\le -2\mu\norm{z}{H\times H}^2-2D(R^\zeta-\norm{z}{H\times H}^\zeta)\norm{z_1}{H}^2,\label{energy-z-3}
\end{align}
where in \eqref{energy-z-2},  we have used that ~$\norm{z_1}{V}\ge\norm{z_1}{H}$.
Observe that~\eqref{energy-z-3} implies that:
\begin{align}
\mbox{If }\norm{z(0)}{H\times H}\le R,\quad\mbox{then }\tfrac{\rmd}{\rmd t}\norm{z}{H\times H}^2
&\le-2\mu\norm{z}{H\times H}^2,\mbox{ for all }t\ge0\notag
\end{align}
which implies~\eqref{exp-mu-z}.
In particular, we see that
\begin{align}
\mbox{If }\norm{z(0)}{H\times H}\le R,\quad\mbox{then }\norm{z(t)}{H\times H}&\le R,\mbox{ for all }t\ge0.\notag
\end{align}
Thus, by~\eqref{energy-z-2}, we find
\begin{align}
\tfrac{\rmd}{\rmd t}\norm{z}{H\times H}^2
+\frac{\nu_2}{2}\norm{z_1}{\rmD(A)}^2+\frac12\norm{z_2}{V}^2+2\mu\norm{(z_1,z_2)}{H\times H}^2
\le0,\notag
\end{align}
from which, after time integration, we obtain~\eqref{energy-z}.
\end{proof}

It follows the aimed stabilizability to the given trajectory~$y_{\ttr}$.

\begin{corollary}\label{C:main}
Let Assumption~\ref{A:NfNg1} hold and let us take actuators and auxiliary functions as in Section~\ref{sS:act}. Let us be given~$\mu>0$ and~$R>0$, and let~$(\overline M,\overline \lambda_1,\overline \lambda_2)=(\overline M,\overline \lambda_1,\overline \lambda_2)(\mu,R)$ be given by Theorem~\ref{T:main}. 
Then, the difference~$(y_1,y_2)=y\coloneqq y_{\ttc}-y_{\ttr}$, between the controlled trajectory solving~\eqref{sys-y-CH-Feed} and the targeted trajectory solving~\eqref{sys-CH}, satisfies
\begin{equation}
\norm{y(t)}{H\times H}^2\le 4(1+\nu_0^2)^2\rme^{-2\mu (t-s)}\norm{y(s)}{H\times H}^2,\quad \mbox{if }\norm{(y_1(0),y_2(0)+\nu_0y_1(0))}{H\times H}\le R.\notag
\end{equation}
\end{corollary}
\begin{proof}
Recalling~\eqref{chvar-y-to-z}, for~$y=y_\ttc-y_\ttr$ we find
\begin{equation}
\begin{split}
\norm{y(t)}{H\times H}^2&=\norm{(y_1(t),y_2(t))}{H\times H}^2=\norm{z_1}{H}^2+\norm{z_2-\nu_0z_1}{H}^2\le (1+2\nu_0^2)\norm{z_1}{H}^2+2\norm{z_2}{H}^2\notag\\
&\le2(1+\nu_0^2)\norm{z(t)}{H\times H}^2.
\end{split}
\end{equation}
With~$(\overline M,\overline \lambda_1,\overline \lambda_2)$ given by Theorem~\ref{T:main} it follows that, for~$M\ge \overline M$, $\lambda_1\ge\overline \lambda_1$, and~$\lambda_2\ge\overline \lambda_2$, the inequality~\eqref{exp-mu-z} holds, giving us
\begin{equation}
\begin{split}
\norm{y(t)}{H\times H}^2&\le 2(1+\nu_0^2)\rme^{-2\mu (t-s)}\norm{z(s)}{H\times H}^2\notag\\
&\hspace{-3em}= 2(1+\nu_0^2)\rme^{-2\mu (t-s)}\norm{(y_1(s),y_2(s)+\nu_0y_1(s))}{H\times H}^2\le4(1+\nu_0^2)^2\rme^{-2\mu (t-s)}\norm{y(s)}{H\times H}^2,\notag
\end{split}
\end{equation}
which finishes the proof.
\end{proof}

%%%%%%%%%%%%%%%%%%%%%%%%
%%%%%%%%%%%%%%%%%%%%%%%%
\subsection{Boundedness of the control input}\label{sS:bddinput}
We show an upper bound for the infinite time-horizon energy of the control input~$\bfu\coloneqq(\bfK_1 y,\bfK_2 y)\in\bbR^{M_\sigma+M_\varsigma}$. By direct computations
\begin{align}
&\norm{\bfu}{L^2(\bbR_+,\bbR^{M_\sigma+M_\varsigma})}^2=\norm{\bfK_1 y}{L^2(\bbR_+,\bbR^{M_\sigma})}^2+\norm{\bfK_2 y}{L^2(\bbR_+,\bbR^{M_\varsigma})}^2\notag\\
&\hspace{1em}=\norm{(U_M^\diamond)^{-1}\bfF_1 z_1}{L^2(\bbR_+,\bbR^{M_\sigma})}^2+\norm{(V_M^\diamond)^{-1}\bfF_2 z_2}{L^2(\bbR_+,\bbR^{M_\varsigma})}^2\le D_{M}^\lambda\norm{z}{L^2(\bbR_+,H\times H)}^2\notag\\
\intertext{with}
&D_{M}^\lambda\coloneqq \lambda_1^2\dnorm{(U_M^\diamond)^{-1}}{}^2\norm{P_{\clU_{M}}^{\widetilde\clU_{M}^\perp}A^2P_{\widetilde\clU_{M}}^{\clU_{M}^\perp}}{\clL(H)}^2 +\lambda_2^2\dnorm{(V_M^\diamond)^{-1}}{}^2\norm{P_{\clV_{M}}^{\widetilde\clV_{M}^\perp} AP_{\widetilde\clV_{M}}^{\clV_{M}^\perp}}{\clL(H)}^2.\notag
\end{align}
with~$\dnorm{(U_M^\diamond)^{-1}}{}^2\coloneqq\norm{(U_M^\diamond)^{-1}}{\clL(H,\bbR^{M_\sigma})}^2$ and~$\dnorm{(V_M^\diamond)^{-1}}{}^2\coloneqq\norm{(V_M^\diamond)^{-1}}{\clL(H,\bbR^{M_\sigma})}^2$. By~\eqref{exp-mu-z},  we obtain the bound 
\begin{align}
\norm{\bfu}{L^2(\bbR_+,\bbR^{M_\sigma+M_\varsigma})}^2
&\le(2\mu)^{-1}D_{M}^\lambda \norm{z(0)}{H\times H}^2<+\infty,\notag
\end{align}
 depending on the norm~$\norm{z(0)}{H\times H}$ of initial difference and on the feedback parameters tuple~$(M,\lambda_1,\lambda_2)$. Setting~$\dnorm{P_{M,1}}{}\coloneqq\norm{P_{\clU_{M}}^{\widetilde\clU_{M}^\perp}}{\clL(H)}$ and~$\dnorm{P_{M,2}}{}\coloneqq\norm{P_{\clV_{M}}^{\widetilde\clV_{M}^\perp}}{\clL(H)}$ we  find that 
\begin{align}\notag
D_{M}^\lambda&\le \lambda_1\dnorm{(U_M^\diamond)^{-1}}{}^2\dnorm{P_{M,1}}{}^2\norm{A^2\rest{\widetilde\clU_{M}}}{\clL(H)}+\lambda_2\dnorm{(V_M^\diamond)^{-1}}{}^2\dnorm{P_{M,2}}{}^2\norm{A\rest{\widetilde\clV_{M}}}{\clL(H)},
\end{align}
thus depending on the  operator norms~$\dnorm{P_{M,1}}{}$ and~$\dnorm{P_{M,2}}{}$  of the oblique projections. These operator norms are bounded and can be computed following~\cite[Cor.~2.9]{KunRod19-cocv} and arguments as in~\cite[Sect.~6]{Rod21-aut}.

%%%%%%%%%%%%%%%%%%%%%%%%%%%
%%%%%%%%%%%%%%%%%%%%%%%%%%%
%%%%%%%%%%%%%%%%%%%%%%%%%%%
\section{Existence and uniqueness of solutions}\label{S:exiuni}
We have shown the exponential stability of system~\eqref{sys-z} (for~$M$, $\lambda_1$, and~$\lambda_2$ large enough) where we have assumed that solutions exist. Here, we show the existence of a suitable class of weak solutions for~\eqref{sys-y-CH-Feed}, under assumptions which include the double-well polynomial potential~\eqref{polyFtau}. Here we consider an arbitrary~$(M,\lambda_1,\lambda_2)\in\bbN_+\times\overline\bbR_+\times\overline\bbR_+$.
Note that the case~$(\lambda_1,\lambda_2)=(0,0)$ corresponds to the free dynamics as in~\eqref{sys-CH}.

\begin{assumption}\label{A:h}
The external force satisfies~$h\in L^2_{\rm loc}(\bbR_+,V'\times V')$.
\end{assumption}

\begin{assumption}\label{A:fgV}
The nonlinearity~$f$ in~\eqref{sys-CH} satisfies~$f=F'$, with~$F\ge0$ and~$F(z)\in L^1(\Omega)$ for all~$z\in\rmD(A^\frac32)$. Further, there are constants~$C_{f*}\ge0$ and~$C_{g*}\ge0$ such that, for all~$v\in \rmD(A)$,
\begin{align}
\norm{f(v)}{V'}^2\le C_{f*}\norm{v}{V}^{2}(1+\norm{v}{H}^2\norm{v}{V}^2)&\quad\mbox{and}\quad
\norm{g(v)}{V'}^2\le C_{g*}\norm{v}{V}^{2}(1+\norm{v}{H}^2\norm{v}{V}^2);\notag\\
-\langle A_0 f(v),v\rangle_{\rmD(A)',\rmD(A)}\le C_{f*}\norm{v}{\rmD(A)}\norm{v}{V}&\quad\mbox{and}\quad -\langle g(t,v),v\rangle_{\rmD(A)',\rmD(A)}\le C_{g*}\norm{v}{\rmD(A)}\norm{v}{V}.\notag
\end{align}
\end{assumption}

\begin{assumption}\label{A:N-dif}
The nonlinearities~$f$ and~$g$ in~\eqref{sys-CH} satisfy, with constants~$\underline C_{f*}\ge0$ and~$\underline C_{g*}\ge0$, and for all~$(v,w)\in \rmD(A^\frac32)\times\rmD(A^\frac32)$,
   \begin{align}  
 \norm{A_0(f(v)- f(w))}{V'}
&\le\underline C_{f*}(1+\norm{(v,w)}{V\times V})(1+\norm{(v,w)}{\rmD(A)\times\rmD(A)})\norm{v- w}{V} \notag\\
\norm{g(v)- g(w)}{V'} &\le \underline C_{g*}(1+\norm{(v,w)}{V\times V})(1+\norm{(v,w)}{\rmD(A)\times\rmD(A)})\norm{v-w}{V}\notag
\end{align}
\end{assumption}

%%%%%%%%%%%%%%%%%%%%%%%%%%%
%%%%%%%%%%%%%%%%%%%%%%%%%%%
\subsection{General bounds for Galerkin approximations}\label{sS:auxbounds-gen}
We have that both targeted and controlled trajectories satisfy a system as
\begin{subequations}\label{sys-y-htilde}
  \begin{align}
 &\tfrac{\p}{\p t}y_{1} +  \nu_2\Delta^2 y_{1}  - \Delta\Bigl(f(y_{1})-\nu_1 y_{2}\Bigr)+g(y_{1})=\widetilde h_1^y,\\
  &\tfrac{\p}{\p t}(y_{2} +\nu_0 y_{1} )- \Delta y_{2}=\widetilde h_2^y,\qquad \clG (y_{})\rest{\p\Omega}= 0,\qquad y(0,\Bigcdot)= y_{0},
\end{align}
\end{subequations}
where~$\widetilde h=(\widetilde h_1,\widetilde h_2)$ represents the sum of external forcings, including the control one, see~\eqref{sys-y-CH-Feed} and~\eqref{Kmonotone},
\begin{align}
& \widetilde h_1(y)=h_1+\chi\bfF_1(\widetilde w_1-w_{r1})\quad\mbox{and}\quad
\widetilde h_2(y)=h_2+\chi\bfF_2(\widetilde w_2-w_{r2}).\notag
\end{align}
with~$w\coloneqq\Xi y=(y_1,y_2+\nu_0y_1)$ and~$\Xi$ as in~\eqref{chvar-y-to-z}.
Here~$\chi\in\{0,1\}$ distinguishes the cases of free and controlled dynamics.
Thus, with
\begin{align}
& \widehat h_1(w_1)=h_1+\chi\bfF_1( w_1-w_{r1})\quad\mbox{and}\quad
\widehat h_2(w_2)=h_2+\chi\bfF_2( w_2-w_{r2}).\notag
\end{align}
it follows
\begin{subequations}\label{sys-w}
  \begin{align}  
 &\dot w_{1} =- \nu_2A_0^2 w_{1} - \nu_0\nu_1A_0 w_{1}+ \nu_1A_0 w_{2} + \Delta f(w_{1})-g(w_{1})+\widehat h_1(w_1),\\
  &\dot w_{2}= -A_0 w_{2}+\nu_0A_0w_{1}+\widehat h_2(w_1).\qquad w(0,\Bigcdot)= w_{0},
\end{align}
\end{subequations}

Let~$\clE_N=\linspan E_N$, where~$E_N$ is the subset of ``the'' first eigenfunctions of the Laplacian.

We consider Galerkin approximations of~\eqref{sys-w}, namely, as follows
\begin{subequations}\label{sys-w-GalN}
  \begin{align}  
 &\dot w^N_{1} =- \nu_2A_0^2 w^N_{1} - \nu_0\nu_1A_0 w^N_{1}+ \nu_1A_0 w^N_{2} + P_{\clE_N}\left(\Delta f(w^N_{1})-g(w^N_{1})+\widehat h_1(w^N_{1})\right),\\
  &\dot w^N_{2}= -A_0 w^N_{2}+\nu_0A_0w^N_{1}+P_{\clE_N}\widehat h_2(w^N_2).\qquad w^N(0,\Bigcdot)= (P_{\clE_N}w_{01},P_{\clE_N}w_{01}),
\end{align}
\end{subequations}
where~$w_{0}\eqqcolon(w_{01},w_{02})$. By testing with ~$2w^N$, we obtain
\begin{align}
\tfrac{\rmd}{\rmd t}\norm{w^N}{H\times H}^2&\le-2\nu_2\norm{A_0w^N_1}{H}^2-2\dnorm{w^N_2}{V_0}^2-2\nu_0\nu_1\dnorm{w^N_1}{V_0}^2+2(\nu_1+\nu_0)\dnorm{w^N_1}{V_0}\dnorm{w^N_2}{V_0}\notag\\
&\quad+2( P_{\clE_N}\Delta f(w^N_{1})-P_{\clE_N}g(w^N_1)+P_{\clE_N}\widehat h_1(w^N_1),w^N_1)_H+2( P_{\clE_N}\widetilde h_2^w,w^N_2)_H,\notag
\end{align}
and by the Young inequality
\begin{align}
2\norm{A_0w^N_1}{H}^2&=2\norm{Aw^N_1-w^N_1}{H}^2=2\norm{Aw^N_1}{H}^2+2\norm{w^N_1}{H}^2-4(Aw^N_1,w^N_1)_H\notag\\
&\ge2\norm{Aw^N_1}{H}^2+2\norm{w^N_1}{H}^2-\norm{Aw^N_1}{H}^2-4\norm{w^N_1}{H}^2=\norm{w^N_1}{\rmD(A)}^2-2\norm{w^N_1}{H}^2,\notag
\end{align}
hence, again by Young inequalities and Assumption~\ref{A:fgV},
\begin{align}
\tfrac{\rmd}{\rmd t}\norm{w^N}{H\times H}^2&\le-\nu_2\norm{w^N_1}{\rmD(A)}^2-\dnorm{w^N_2}{V_0}^2+C_1\norm{w^N_1}{V}^2\notag\\
&\quad-2(  P_{\clE_N}A_0 f(w^N_{1})-P_{\clE_N}g(w^N_1)+P_{\clE_N}\widehat h_1(w^N_1),w^N_1)_H+2( P_{\clE_N} \widehat h_2(w^N_2),w^N_2)_H\notag\\
&\hspace{-4em}\le-\nu_2\norm{w^N_1}{\rmD(A)}^2-\dnorm{w^N_2}{V_0}^2+C_1\norm{w^N_1}{V}^2\notag\\
&\hspace{-3em}+2(C_{f*}+C_{g*})\norm{w^N_1}{\rmD(A)}\norm{w^N_1}{V}+2( P_{\clE_N} \widehat h_1(w^N_1),w^N_1)_H+2( P_{\clE_N} \widehat h_2(w^N_2),w^N_2)_H\notag\\
&\hspace{-4em}\le-\frac{\nu_2}2\norm{w^N_1}{\rmD(A)}^2-\dnorm{w^N_2}{V_0}^2+C_2\norm{w^N_1}{V}^2+2(  P_{\clE_N}\widehat h_1(w^N_1),w^N_1)_H+2( P_{\clE_N} \widehat h_2(w^N_2),w^N_2)_H,\notag
\end{align}
with~$C_2$ independent of~$N$. From Young inequalities and
\begin{align}
& P_{\clE_N}\widehat h_j(w^N_j)=P_{\clE_N}h_j+ P_{\clE_N}\chi\bfF_j(w_j^N-w_{\ttr j}),\qquad j\in\{1,2\},\notag
\end{align}
it follows that
\begin{align}
\tfrac{\rmd}{\rmd t}\norm{w^N}{H\times H}^2&\le-\frac{\nu_2}4\norm{w^N_1}{\rmD(A)}^2-\frac12\norm{w^N_2}{V}^2+C_2\norm{w^N_1}{V}^2+C_3\norm{(h_1, h_2)}{\rmD(A)'\times V'}^2+C_3\norm{w^N_2}{H}^2\notag\\
&\quad\!\!+C_3\norm{\chi\bfF_1(w_1^N-w_{\ttr 1}),w^N_1)}{H}^2+C_3\norm{\chi\bfF_2(w_2^N-w_{\ttr 2}),w^N_1)}{H}^2.\notag\\
&\le-\frac{\nu_2}4\norm{w^N_1}{\rmD(A)}^2-\frac12\norm{w^N_2}{V}^2+C_4\norm{w^N_1}{V}^2+C_4\norm{w^N_2}{H}^2\notag\\
&\quad\!\!+C_3\norm{h}{\rmD(A)'\times V'}^2+C_4\chi\norm{w_{\ttr}}{H\times H}^2,\notag
\end{align}
with~$C_3$ and~$C_4$ independent of~$N$. 
By interpolation, $\norm{w^N_1}{V}\le C_5\norm{w^N_1}{\rmD(A)}\norm{w^N_1}{H}$ and
\begin{align}
\frac{\nu_2}8\norm{w^N_1}{\rmD(A)}^2+\frac12\norm{w^N_2}{V}^2+\tfrac{\rmd}{\rmd t}\norm{w^N}{H\times H}^2&\le C_6\norm{w^N}{H\times H}^2+C_3\norm{h}{\rmD(A)'\times V'}^2+\chi C_4\norm{w_{\ttr}}{H\times H}^2,\notag
\end{align}
with~$C_6$ independent of~$N$. By Assumption~\ref{A:h}, the Gronwall lemma, and time integration it follows that
\begin{align}\label{weakbdd-GalN}
\|w^N\|_{L^\infty((0,T),H\times H){\,\textstyle\bigcap\,}L^2((0,T),\rmD(A)\times V)}\le C_7,
\end{align}
with~$C_7$ independent of~$N$ (firstly for the case~$\chi=0$ and afterwards for the case~$\chi=1$).

%%%%%%%%%%%%%%%%%%%%%%%%%%%
%%%%%%%%%%%%%%%%%%%%%%%%%%%
\subsection{Further bounds for Galerkin approximations}\label{sS:auxbounds-part}
We shall need more regularity for the solutions. For this purpose we  introduce~$\fkp^N$ (resembling a chemical potential; cf.~\cite[Eq.~(1.3)]{BarbuColliGilMar17}\cite[Eq.~(2)]{BrunkEggerHabrich23}\cite[Sect.2]{Fritz24}) as follows
\begin{equation}\label{fkp}
\fkp^N\coloneqq \nu_2A_0 w_1^N +\nu_0\nu_1 w_1^N+P_{\clE_N}f(w_1^N).
\end{equation}
Testing the dynamics of~$w^N_1$ with~$2\nu_2Aw^N_1$, we find the estimates
\begin{align}
&\nu_2\tfrac{\rmd}{\rmd t}\norm{w^N_1}{V}^2=2\nu_2(\dot w^N_1,Aw^N_1)_H\notag\\
&\hspace{0em}=2( -A_0\fkp^N+\nu_1A_0w^N_2-P_{\clE_N}g(w^N_1)+P_{\clE_N}\widehat h_1(w^N_1),\nu_2Aw^N_1)_H\notag\\
&\hspace{0em}=2( -A_0\fkp^N+\nu_1A_0w^N_2-P_{\clE_N}g(w^N_1)+P_{\clE_N}\widehat h_1(w^N_1),\fkp^N-\nu_0\nu_1 w^N_1-P_{\clE_N}f(w^N_1)+\nu_2w^N_1)_H\notag\\
&\hspace{0em}=-2\dnorm{\fkp^N}{V_0}^2+2( A_0\fkp^N,\nu_0\nu_1 w^N_1+P_{\clE_N}f(w^N_1)+\nu_2w^N_1)_H\notag\\
&\hspace{1em}+2( \nu_1A_0w^N_2-P_{\clE_N}g(w^N_1)+P_{\clE_N}\widehat h_1(w^N_1),\fkp^N-\nu_0\nu_1 w^N_1-P_{\clE_N}f(w^N_1)+\nu_2w^N_1)_H,\notag
\end{align}
and, recalling Assumption~\ref{A:fgV} and denoting~$\langle F(w^N_1),1\rangle\coloneqq\int_\Omega F(w^N_1)\,\rmd\Omega$,
\begin{align}
2\tfrac{\rmd}{\rmd t}\langle  F(w^N_1),1\rangle&=2( f(w^N_1),\dot w^N_1)_H=2( P_{\clE_N}f(w^N_1),\dot w^N_1)_H\notag\\
&=2( P_{\clE_N} f(w^N_1), -A_0\fkp^N+\nu_1A_0w^N_2-P_{\clE_N}g(w^N_1)+P_{\clE_N}\widehat h_1(w^N_1))_H.\notag
\end{align} 
Therefore, we obtain
\begin{align}
\tfrac{\rmd}{\rmd t}(\nu_2\norm{w^N_1}{V}^2+2\langle F(w^N_1),1\rangle)&=-2\dnorm{\fkp^N}{V_0}^2+2\nu_0\nu_1( A_0\fkp^N, w^N_1)_H\notag\\
&\hspace{1em}
+2(\nu_1 A_0w^N_2-P_{\clE_N}g(w^N_1)+P_{\clE_N}\widehat h_1(w^N_1),\fkp^N-\nu_0\nu_1 w^N_1)_H\notag\\
&\hspace{1em}+2( A_0\fkp^N+\nu_1A_0w^N_2-P_{\clE_N}g(w^N_1)+P_{\clE_N}\widehat h_1(w^N_1),\nu_2w^N_1)_H,\notag
\end{align} 
and, with~$r_{gh}\coloneqq P_{\clE_N}g(w^N_1)-P_{\clE_N}\widehat h_1(w^N_1)$,
\begin{align}
&\tfrac{\rmd}{\rmd t}(\nu_2\norm{w^N_1}{V}^2+2\langle F(w^N_1),1\rangle)\notag\\
&=-2\dnorm{\fkp^N}{V_0}^2+2( A_0\fkp^N, \nu_0\nu_1w^N_1+\nu_1w^N_2+\nu_2w^N_1)_H
-2( r_{gh},\fkp^N-\nu_0\nu_1 w^N_1)_H\notag\\
&\hspace{1em}-2( \nu_1A_0w^N_2,\nu_0\nu_1w^N_1)_H+2( \nu_1A_0w^N_2-r_{gh},\nu_2w^N_1)_H\notag\\
&\hspace{0em}\le-(2-\gamma)\dnorm{\fkp^N}{V_0}^2+\gamma^{-1}\dnorm{\nu_0\nu_1w^N_1+\nu_1w^N_2+\nu_2w^N_1}{V_0}^2+(2+\gamma^{-1})\norm{r_{gh}}{V'}^2+\gamma\norm{\fkp^N}{V}^2\notag\\
&\hspace{1em}+(\nu_0\nu_1)^2\norm{w^N_1}{V}^2+\nu_2^2\norm{w^N_1}{H}^2+\nu_1^2\dnorm{w^N_2}{V_0}^2+(\nu_0\nu_1-\nu_2)^2\dnorm{w^N_1}{V_0}^2,\notag
\end{align}
and, using~$\norm{\fkp^N}{V}^2=\dnorm{\fkp^N}{V_0}^2+\norm{\fkp^N}{H}^2$, gives
\begin{align}
&\tfrac{\rmd}{\rmd t}(\nu_2\norm{w^N_1}{V}^2+2\langle F(w^N_1),1\rangle)\notag\\
&\le-(2-2\gamma)\dnorm{\fkp^N}{V_0}^2+(2+\gamma^{-1})\norm{r_{gh}}{V'}^2+(C_8+C_9\gamma^{-1})\dnorm{w^N_1}{V}^2+C_8\dnorm{w^N_2}{V}^2+\gamma\norm{\fkp^N}{H}^2,\notag
\end{align}
with~$C_8$ and~$C_9$ independent of~$N$ and~$\gamma$. 
Next, for every~$\varrho>0$ and all~$\gamma>0$,
\begin{align}
\varrho\tfrac{\rmd}{\rmd t}\norm{w^N_2}{H}^2&=2\varrho(\dot w^N_2,w^N_2)_H\notag\\
&=-2\varrho( A_0 w^N_{2},w^N_2)_H +2\nu_0\varrho( A_0w^N_{1},w^N_2)_H+ 2\varrho( P_{\clE_N}\widehat h_2(w^N_2),w^N_2)_H\notag\\
&\le-(2-\gamma)\varrho\norm{w^N_2}{V_0}^2 +\varrho\gamma^{-1}\nu_0^2\norm{w^N_1}{V_0}^2+ \gamma\varrho\norm{w^N_2}{V}^2+\gamma^{-1}\varrho\norm{\widehat h_2(w^N_2)}{V'}^2\notag\\
&\le-(2-2\gamma)\varrho\norm{w^N_2}{V_0}^2 +\varrho\gamma^{-1}\nu_0^2\norm{w^N_1}{V_0}^2+ \gamma\varrho\norm{w^N_2}{H}^2+\gamma^{-1}\varrho\norm{\widehat h_2(w^N_2)}{V'}^2\notag.
\end{align}

Hence, choosing~$\gamma=\frac12$, we have that
\begin{align}
&\tfrac{\rmd}{\rmd t}(\nu_2\norm{w^N_1}{V}^2+2\langle F(w^N_1),1\rangle+\varrho\norm{w^N_2}{H}^2)\notag\\
&\hspace{4em}\le -\dnorm{\fkp^N}{V_0}^2+4\norm{r_{gh}}{V'}^2+C_{10}\norm{w^N_1}{V}^2+C_{10}\norm{w^N_2}{V}^2+\tfrac12\norm{\fkp^N}{H}^2\notag\\
&\hspace{5em}-\varrho\dnorm{w^N_2}{V_0}^2+\varrho C_{10}\dnorm{w^N_1}{V_0}^2+ \tfrac12\varrho\norm{w^N_2}{H}^2+2\varrho\norm{\widehat h_2(w^N_2)}{V'}^2,\notag
\end{align}
with~$C_{10}$ independent of~$\varrho$ and~$N$. Choosing now~$\varrho=1+C_{10}$, 
it follows that
\begin{align}
&\dnorm{\fkp^N}{V_0}^2+\dnorm{w^N_2}{V_0}^2+\tfrac{\rmd}{\rmd t}(\nu_2\dnorm{w^N_1}{V_0}^2+2\langle F(w^N_1),1\rangle+\varrho\norm{w^N_2}{H}^2)\notag\\
&\hspace{2em}\le 4\norm{r_{gh}}{V'}^2+C_{11}\norm{w^N_1}{V}^2+C_{11}\norm{w^N_2}{H}^2+\tfrac12\norm{\fkp^N}{H}^2 +C_{11}\norm{\widehat h_2(w^N_2)}{V'}^2,\notag
\end{align}
with~$C_{11}$ independent of~$\varrho$ and~$N$. 
Denoting, by simplicity, the ``free energy'' functional (a variation of~\cite[Eq.~(4)]{Caginalp88})
\begin{align}\label{fkW1}
\fkW(w^N)\coloneqq\nu_2\norm{w^N_1}{V}^2+2\langle F(w^N_1),1\rangle+\varrho\norm{w^N_2}{H}^2,\end{align}
we find
\begin{align}
&\dnorm{\fkp^N}{V_0}^2+\dnorm{w^N_2}{V_0}^2+\tfrac{\rmd}{\rmd t}\fkW(w^N)\le C_{12}\left(\norm{w^N}{V\times H}^2+\norm{\fkp^N}{H}^2+\norm{g(w^N_1)}{V'}^2+\norm{\widehat h(w^N)}{V'\times V'}^2\right).\notag
\end{align}
From~\eqref{fkp} and the Young inequality we find, for all~$\eta>0$,
\begin{align}
2\norm{\fkp^N}{H}^2&=2(\nu_2A_0 w^N_1 +\nu_0\nu_1 w^N_1+P_{\clE_N}f(w^N_1),\fkp^N)_H\notag\\
&\le2\nu_2\dnorm{w^N_1}{V_0}\dnorm{\fkp^N}{V_0}+2\nu_0\nu_1\norm{w^N_1}{H}\norm{\fkp^N}{H}+2\norm{ f(w^N_1)}{V'}\norm{\fkp^N}{V}\notag\\
&\le 2\eta\dnorm{\fkp^N}{V_0}^2 +\norm{\fkp^N}{H}^2 +(\eta^{-1} \nu_2^2+2(\nu_0\nu_1)^2)\norm{w^N_1}{V}^2+(\eta^{-1}+2)\norm{f(w^N_1)}{V'}^2,\notag
\end{align}
which leads to
\begin{align}
&\dnorm{\fkp^N}{V_0}^2+\dnorm{w^N_2}{V_0}^2+\tfrac{\rmd}{\rmd t}\fkW(w^N)\notag\\
&\hspace{3em}\le 2\eta C_{12}\dnorm{\fkp^N}{V_0}^2+C_{12}(1+ \eta^{-1} \nu_2^2+2(\nu_0\nu_1)^2))\norm{w^N_1}{V}^2 \notag\\
&\hspace{4em}+C_{12}\left(\norm{w^N_2}{H}^2+\norm{g(w^N_1)}{V'}^2+\norm{\widehat  h(w^N)}{V'\times V'}^2\right)
+C_{12}(\eta^{-1}+2)\norm{f(w^N_1)}{V'}^2.
\notag
\end{align}
By Assumption~\ref{A:fgV} and~\eqref{weakbdd-GalN}, we have that
\begin{align}
&\dnorm{\fkp^N}{V_0}^2+\dnorm{w^N_2}{V_0}^2+\tfrac{\rmd}{\rmd t}\fkW(w^N)\notag\\
&\hspace{3em}\le 2\eta C_{12}\dnorm{\fkp^N}{V_0}^2+C_{13}(1+ \eta^{-1})\norm{w^N_1}{V}^2 +C_{13}(1+ \eta^{-1})\norm{w^N_1}{H}^2\norm{w^N_1}{V}^4 \notag\\
&\hspace{4em}+C_{12}\left(\norm{w^N_2}{H}^2+\norm{\widehat  h(w^N)}{V'\times V'}^2\right)\notag\\
&\hspace{3em}\le 2\eta C_{12}\dnorm{\fkp^N}{V_0}^2+C_{13}(1+ \eta^{-1})\norm{w^N_1}{V}^2 +C_{14}(1+ \eta^{-1})\norm{w^N_1}{V}^4 \notag\\
&\hspace{4em}+C_{12}\left(\norm{w^N_2}{H}^2+\norm{\widehat  h(w^N)}{V'\times V'}^2\right).
\notag
\end{align}

Now, we choose~$\eta=(4C_{12})^{-1}$ and obtain
\begin{align}
&\tfrac12\dnorm{\fkp^N}{V_0}^2+\dnorm{w_2^N}{V_0}^2+\tfrac{\rmd}{\rmd t}\fkW(w^N)\notag\\
&\qquad\le C_{15}\left(\norm{w_1^N}{V}^2+\norm{w_1^N}{V}^4+\norm{w_2^N}{H}^2+\norm{\widehat  h(w^N)}{V'\times V'}^2\right)\notag\\
&\qquad\le C_{15}\left((1+\norm{w_1^N}{V}^2)\norm{w_1^N}{V}^2+\norm{w_2^N}{H}^2+\norm{\widehat  h(w^N)}{V'\times V'}^2\right).\label{bddGal}
\end{align}

Now, by using
\begin{align}
& \widehat h_j(w^N_j)=h_j+ \chi\bfF_j(w_j^N-w_{\ttr j}),\qquad j\in\{1,2\},\label{hath}\\
& \norm{\widehat  h(w^N)}{V'\times V'}^2\le 2\norm{  h}{V'\times V'}^2+2\chi C_{16}\norm{w^N}{H\times H}^2+2\chi C_{16}\norm{w_\ttr}{H\times H}^2,\notag
\end{align}
where we have also used~$\norm{z}{V'}\le\norm{z}{H}$. Note that we also have that~$\norm{z}{H}\le\norm{z}{V}$ and that,
by Assumption~\ref{A:fgV} we have~$\langle F(\zeta),1\rangle$ as in~\eqref{fkW1} is nongegative. These facts, together with estimate~\eqref{bddGal} and~\eqref{weakbdd-GalN} lead to
\begin{align}
\tfrac12\dnorm{\fkp^N}{V_0}^2+\dnorm{w^N_2}{V_0}^2+\tfrac{\rmd}{\rmd t}\fkW(w)&\le C_{17}(1+\norm{w_1}{V}^2)\fkW(w)+C_{17}(1+\norm{  h}{V'\times V'}^2).\notag
\end{align}
In particular, by the Gronwall inequality, \eqref{weakbdd-GalN}, and time integration, it follows that
\begin{align}\label{strgbdd-GalNV}
\norm{w^N}{L^\infty((0,T),V)}\le C_{18}\quad\mbox{and}\quad \int_0^T\dnorm{\fkp^N(t)}{V_0}^2\,\rmd t \le C_{18},
\end{align}
with~$C_{18}$ independent of~$N$.

Testing the dynamics of~$w^N_1$ with~$2Aw^N_1$, with~$R_{fg}\coloneqq P_{\clE_N}\Delta f(w_{1}^N)-P_{\clE_N}g(w_{1}^N)$, we find
 \begin{align}  
 \tfrac{\rmd}{\rmd t}\norm{w_{1}^N}{V}^2 &=2(- \nu_2 A_0^2 w_{1}^N - \nu_0\nu_1A_0 w_{1}^N+ \nu_1A_0 w_{2}^N +R_{fg}+P_{\clE_N}\widehat h_1(w_1^N),Aw^N_1)_H\notag
 \end{align}
 with
  \begin{align}  
2(- \nu_2 A_0^2 w_{1}^N,Aw^N_1)_H &=-2\nu_2\norm{w_{1}^N}{\rmD(A^\frac32)}^2 +2\nu_2((A^2-A_0^2)w_{1}^N,Aw^N_1)_H\notag\\
&=-2\nu_2\norm{w_{1}^N}{\rmD(A^\frac32)}^2 +2\nu_2((2A_0+\Id)w_{1}^N,Aw^N_1)_H\notag\\
&=-2\nu_2\norm{w_{1}^N}{\rmD(A^\frac32)}^2 +2(2\nu_2A_0w_{1}^N,Aw^N_1)_H+2\nu_2\norm{w_{1}^N}{V}^2.\notag
\end{align}
Hence
 \begin{align}  
 \tfrac{\rmd}{\rmd t}\norm{w_{1}^N}{V}^2 &=-2\nu_2\norm{w_{1}^N}{\rmD(A^\frac32)}^2+2\nu_2\norm{w_{1}^N}{V}^2+2((2\nu_2- \nu_0\nu_1)A_0 w_{1}^N+ \nu_1A_0 w_{2}^N,Aw^N_1)_H \notag\\&\quad+2(R_{fg}+P_{\clE_N}\widehat h_1(w_1^N),Aw^N_1)_H\notag
 \end{align}
 and, using again~$A_0=A-\Id$ and the Young inequality,
 \begin{align}  
  \tfrac{\rmd}{\rmd t}\norm{w_{1}^N}{V}^2&\le-2\nu_2\norm{w_{1}^N}{\rmD(A^\frac32)}^2+C_{20}\norm{w_{1}^N}{V}^2+C_{20}\norm{w_{2}^N}{V}^2\notag\\
 &\quad+2C_{20}(\norm{w_{1}^N}{V}+\norm{w_{2}^N}{V})\norm{w^N_1}{\rmD(A^\frac32)} +2(R_{fg}+P_{\clE_N}\widehat h_1(w_1^N),Aw^N_1)_H\notag\\
 &\le-(2\nu_2-4\gamma)\norm{w_{1}^N}{\rmD(A^\frac32)}^2+C_{20}\norm{w_{1}^N}{V}^2+C_{20}\norm{w_{2}^N}{V}^2\notag\\
 &\quad+C_{20}^2\gamma^{-1}(\norm{w_{1}^N}{V}^2+\norm{w_{2}^N}{V}^2) +\gamma^{-1}(\norm{R_{fg}}{V'}^2+\norm{P_{\clE_N}\widehat h_1(w_1^N)}{V'}^2)\notag
 \end{align}
for all~$\gamma>0$. Choosing~$\gamma=\frac{\nu_2}4$ and recalling~\eqref{hath},  we obtain
 \begin{align}  
  \tfrac{\rmd}{\rmd t}\norm{w_{1}^N}{V}^2
 &\le-\nu_2\norm{w_{1}^N}{\rmD(A^\frac32)}^2+C_{21}\norm{w_{1}^N}{V}^2+C_{21}\norm{w_{2}^N}{V}^2\notag\\
 &\quad+C_{21}(\norm{\Delta f(w_{1}^N)-g(w_{1}^N)}{V'}^2+\norm{h_1}{V'}^2+\chi\norm{w_{1}^N}{H}^2),\notag
 \end{align}
 with~$\chi\in\{0,1\}$. Now, we use Assumption~\ref{A:fgV}, together with  Young and interpolation inequalities, to obtain
  \begin{align}  
  \nu_2\norm{w_{1}^N}{\rmD(A^\frac32)}^2+\tfrac{\rmd}{\rmd t}\norm{w_{1}^N}{V}^2
 &\le C_{22}\norm{w_{1}^N}{V}^2+C_{21}\norm{w_{2}^N}{V}^2+C_{22}\norm{w_{1}^N}{\rmD(A)}^2+C_{21}\norm{h_1}{V'}^2\notag
 \end{align}
 where we used again~$\norm{w_{1}^N}{H}\le\norm{w_{1}^N}{V}$. By~\eqref{weakbdd-GalN}, \eqref{strgbdd-GalNV}, and time integration, we find that
 \begin{align}\label{strgbdd-GalNDA}
\norm{w^N}{L^2((0,T),\rmD(A^\frac32))}\le C_{23}
\end{align}
 with~$C_{23}$ independent of~$N$.
Next, by the dynamics in~\eqref{sys-w-GalN}, Assumption~\ref{A:fgV} and~\eqref{hath},
  \begin{align}  
 \norm{\dot w^N_{1}}{V'} &\le C_{23}\norm{ w^N_{1} }{\rmD(A^\frac32)}+C_{23}\norm{ w^N_{2} }{V} + \norm{\Delta f(w^N_{1})-g(w^N_{1})+\widehat h_1(w^N_{1})}{V'}\notag\\
 &\le C_{23}\norm{ w^N_{1} }{\rmD(A^\frac32)}+C_{24}\norm{ w^N_{2} }{V} + C_{24}\norm{w^N_{1}}{\rmD(A)}+ C_{24}\norm{h_1}{V'},\notag\\
  \norm{\dot w^N_{2}}{V'}  &\le C_{25}\norm{ w^N_{2} }{V}+C_{25}\norm{ w^N_{1} }{V}+ \norm{\widehat h_2(w^N_2)}{V'}\le C_{26}\norm{ w^N_{2} }{V}+C_{25}\norm{ w^N_{1} }{V}+ C_{26}\norm{h_2}{V'}.\notag
\end{align}

By~\eqref{weakbdd-GalN}, \eqref{strgbdd-GalNV}, and~\eqref{strgbdd-GalNDA}, it follows that
 \begin{align}\notag
\norm{\dot w^N}{L^2((0,T),V'\times V')}\le C_{27},
\end{align}
 with~$C_{27}$ independent of~$N$. In conclusion, we have that
  \begin{align}\label{strgbdd-GalN-Wspace}
\norm{w^N}{W((0,T),\rmD(A^\frac32)\times V,V'\times V')}\le C_{28},
\end{align}
  with~$C_{28}$ independent of~$N$

%%%%%%%%%%%%%%%%%%%%%%%%%%%
%%%%%%%%%%%%%%%%%%%%%%%%%%%
\subsection{Existence}\label{sS:exist}
We show that solutions~$w$ for~\eqref{sys-w} exist in~$W_{\rm loc}(\bbR_+,\rmD(A^\frac32)\times V,V'\times V')\eqqcolon\bfW$, for initial states~$z(0)\in V\times H$. Note that, this implies the existence of solutions~$y\coloneqq\Xi^{-1} w$ for~\eqref{sys-y-htilde} in the same space~$\bfW$,  with~$\Xi$ defined  in~\eqref{chvar-y-to-z}. As usual, we use the uniform bound~\eqref{strgbdd-GalN-Wspace} for the sequence of Galerkin approximations~$(w^N)_{N\in\bbN_+}$, and take a weak limit~$w^\infty\in \bfW$ of a susequence of~$(w^N)_{N\in\bbN_+}$ (which we still denote by~$(w^N)_{N\in\bbN_+}$, for simplicity)
  \begin{align} \label{weak-lim}
 w^N \xrightharpoonup[\bfW]{}w^\infty.
  \end{align}
To show that~$w^\infty$ solves~\eqref{sys-w}, it is enough to show that the terms in the dynamics equation converge weakly in~$L^2((0,T), V')$. For the affine terms (recall~\eqref{hath}) we have
  \begin{align}  
 &-\dot w^N_{1} - \nu_2A_0^2 w^N_{1} - \nu_0\nu_1A_0 w^N_{1}+ \nu_1A_0 w^N_{2} + P_{\clE_N}\widehat h_1(w^N_{1})\notag\\&\hspace{5em}\xrightharpoonup[L^2((0,T), V')]{}-\dot w_{1}^\infty - \nu_2A_0^2 w_{1}^\infty - \nu_0\nu_1A_0 w_{1}^\infty+ \nu_1A_0 w_{2}^\infty + \widehat h_1(w_{1}^\infty),\notag\\
  &-\dot w^N_{2} -A_0 w^N_{2}+\nu_0A_0w^N_{1}+P_{\clE_N}\widehat h_2(w^N_2)\notag\\&\hspace{5em}\xrightharpoonup[L^2((0,T), V')]{}-\dot w_{2} -A_0 w_{2}^\infty+\nu_0A_0w_{1}^\infty+\widehat h_2(w_2^\infty),\notag
\end{align}
Therefore, it remains to show that the nonlinear terms also converge weakly
  \begin{align}  
 & P_{\clE_N}\Delta f(w^N_{1})-P_{\clE_N}g(w^N_{1})\xrightharpoonup[L^2((0,T), V')]{}\Delta f(w_{1}^\infty)-g(w_{1}^\infty).\label{goal-convN}
\end{align}
For this purpose, we denote~$V_T\coloneqq L^2((0,T), V)$,   and observe that for any~$\varphi\in V_T$,
  \begin{align}  
& \langle P_{\clE_N}A_0 f(w^N_{1})-A_0  f(w^\infty_{1}),\varphi\rangle_{V_T', V_T}\notag\\
 &\hspace{1em}=\langle P_{\clE_N}(A_0  f(w^N_{1})-A_0  f(w^\infty_{1})),\varphi\rangle_{ V_T',V_T}+\langle P_{\clE_N}A_0 f(w^\infty_{1})-A_0  f(w^\infty_{1}),\varphi\rangle_{V_T', V_T}\notag\\
  &\hspace{1em}\le\norm{A_0(f(w^N_{1})- f(w^\infty_{1}))}{V_T'}\norm{\varphi}{V_T}+\norm{A_0 f(w^\infty_{1})}{V_T'}\norm{(1-P_{\clE_N})\varphi}{V_T};\notag
 \intertext{and}
 & \langle P_{\clE_N}g(w^N_{1})-g(w^\infty_{1}),\varphi\rangle_{V_T', V_T}\notag\\
 &\hspace{1em}=\langle P_{\clE_N}(g(w^N_{1})-g(w^\infty_{1})),\varphi\rangle_{V_T', V_T}+\langle P_{\clE_N}g(w^\infty_{1})-g(w^\infty_{1}),\varphi\rangle_{V_T', V_T}\notag\\
   &\hspace{1em}\le\norm{g(w^N_{1})- g(w^\infty_{1})}{V_T'}\norm{\varphi}{V_T}+\norm{g(w^\infty_{1})}{V_T'}\norm{(1-P_{\clE_N})\varphi}{V_T};\notag
\end{align}
 which, together with Assumption~\ref{A:N-dif}, lead to
   \begin{align}  
\Theta\coloneqq&\lim_{N\to+\infty} \left(\langle P_{\clE_N}A_0 f(w^N_{1})-A_0  f(w^\infty_{1}),\varphi\rangle_{V_T', V_T} +\langle P_{\clE_N}g(w^N_{1})-g(w^\infty_{1}),\varphi\rangle_{V_T', V_T}\right)\notag\\
&\hspace{3em}\le\lim_{N\to+\infty}\left(\norm{A_0(f(w^N_{1})- f(w^\infty_{1}))}{V_T'}\norm{\varphi}{V_T}+\norm{g(w^N_{1})- g(w^\infty_{1})}{V_T'}\norm{\varphi}{V_T}\right)\notag\\
&\hspace{3em}\le C_1\norm{\varphi}{V_T}\lim_{N\to+\infty}
\norm{\Bigl(1+\norm{(w^N_{1},w^\infty_{1})}{\rmD(A)\times\rmD(A)}\Bigr)
\norm{w^N_{1}- w^\infty_{1}}{V}}{L^2((0,T), \bbR)}.\notag
\end{align}
By interpolation, \eqref{strgbdd-GalNV}, \eqref{strgbdd-GalN-Wspace}, and H\"older inequality,
   \begin{align}  
\Theta&\le C_1\norm{\varphi}{V_T}\lim_{N\to+\infty}\norm{w^N_{1}- w^\infty_{1}}{L^2((0,T),V)}\notag\\
&\quad+C_2\norm{\varphi}{V_T}\norm{\norm{(w^N_{1},w^\infty_{1})}{\rmD(A^{\frac32})\times\rmD(A^{\frac32})}\norm{w^N_{1}- w^\infty_{1}}{V}}{L^1((0,T), \bbR)}\notag\\
&\le C_3\norm{\varphi}{V_T}\lim_{N\to+\infty}\norm{w^N_{1}- w^\infty_{1}}{L^2((0,T), V)}.\notag
\end{align}
By~\eqref{weak-lim} and~$\bfW\xhookrightarrow{\rmc} L^2((0,T), V)$ (e.g., see~\cite[Ch.~3, Sect.~2.2, Thm.~2.1]{Temam01}), we can assume that~$w^N_{1} \to w^\infty_{1}$ strongly in~$L^2((0,T), V)$; by taking a subsequence, if necessary. Thus, we arrive at~\eqref{goal-convN}.

%%%%%%%%%%%%%%%%%%%%%%%%%%%
%%%%%%%%%%%%%%%%%%%%%%%%%%%
\subsection{Uniqueness}\label{sS:unique}
We show that solutions~$w$  for~\eqref{sys-w} in~$W_{\rm loc}(\bbR_+,\rmD(A^\frac32)\times V,V'\times V')=\bfW$ are unique, which implies the uniqueness of solutions~$y\coloneqq\Xi^{-1} w$ for~\eqref{sys-y-htilde} in~$\bfW$.

We know that a solution~$w=w^\infty\in\bfW$ does exist. Let~$\widetilde w\in \bfW$ be another solution and fix an arbitrary~$T>0$. The difference~$d\coloneqq \widetilde w-w$ solves, for~$t\in(0,T)$, the analogue of~\eqref{sys-z-evol},
\begin{align}\label{sys-d}
\dot d&=-\bfA d-\bfA_{\rm rc} d-\clN(\widetilde w_1)+\clN(w_1)+\widetilde h_1^{\widetilde w}-\widetilde h_1^w,\qquad d(0)= 0,
\end{align}
with~$\clN(z)\coloneqq A_0f(z_1)+g(z_1)$.
By Assumption~\ref{A:N-dif} and~$W((0,T),\rmD(A^\frac32),V')\xhookrightarrow{}\clC([0,T],V)$, it follows that, for some constant~$C_0>0$,
\begin{align}
-2\langle\clN(\widetilde w_1)-\clN(w_1), d\rangle_{\rmD(A)',\rmD(A)}&\le C_0(1+\norm{(\widetilde w_1,w_1)}{\rmD(A)\times\rmD(A)})\norm{d_1}{V}^2.\notag
\end{align}
Then, multiplying~\eqref{sys-d} by~$2d$, and using~$\rmD(A^\frac32)\xhookrightarrow{}\rmD(A)$,
\begin{align}
\tfrac{\rmd}{\rmd t}\norm{d}{H\times H}^2&\le-2\nu_2\norm{d_1}{\rmD(A)}^2-2\norm{d_2}{V}^2-2(\bfA_{\rm rc} d+\widetilde h_1^{\widetilde w}-\widetilde h_1^w,d)_{H\times H}\notag\\
&\quad+ C_0(1+\clQ)\norm{d_1}{V}^2,\quad\mbox{with }\clQ\coloneqq\norm{(\widetilde w_1,w_1)}{\rmD(A^\frac32)\times\rmD(A^\frac32)}.\notag
\end{align}
Recalling the definition of~$\bfA_{\rm rc}$ in ~\eqref{sys-z-evol}, it follows that
\begin{align}
\tfrac{\rmd}{\rmd t}\norm{d}{H\times H}^2&\le-2\nu_2\norm{d_1}{\rmD(A)}^2-2\norm{d_2}{V}^2+C_1\norm{d_1}{V}^2+C_0\clQ\norm{d_1}{V}^2\notag\\
&\quad+C_2\dnorm{d_1}{V_0}\dnorm{d_2}{V_0}+2\norm{d_2}{H}-2(\zeta,d)_{H\times H},\notag
\end{align}
with~$(\zeta_1,\zeta_2)=\zeta\coloneqq\widetilde h_1^{\widetilde w}-\widetilde h_1^w$, for which we find recalling~\eqref{OP-FeedK},
\begin{align}
&\zeta_1=(h_1+\bfF_1(\widetilde w_1-w_{r1}))-(h_1+\bfF_1(w_1-w_{r1})=\bfF_1(d_1),\notag\\
&\zeta_2=(h_2+\bfF_2(\widetilde w_2-w_{r2}))-(h_2+\bfF_2(w_2-w_{r2}))=\bfF_2(d_2).\notag
\end{align}
By~\eqref{Kmonotone}, we find~$-2(\zeta,d)_{H\times H}=-2(\bfF_1(d_1),d_1)_{H}-2(\bfF_2(d_2),d_2)_{H}\le0$ and by Young and interpolation inequalities
\begin{align}
\tfrac{\rmd}{\rmd t}\norm{d}{H\times H}^2&\le-2\nu_2\norm{d_1}{\rmD(A)}^2-\norm{d_2}{V}^2+C_3\norm{d_1}{V}^2+2\norm{d_2}{H}^2+C_0\clQ\norm{d_1}{V}^2.\notag
\end{align}
Since~$\clQ$ is integrable in the time interval~$(0,T)$ and~$d(0)=0$, by  the Gronwall inequality it follows that~$d(t)=0$, $t\in(0,T)$.

%%%%%%%%%%%%%%%%%%%%%%%%%%%%
%%%%%%%%%%%%%%%%%%%%%%%%%%%%
%%%%%%%%%%%%%%%%%%%%%%%%%%%%
\section{On the assumptions on the nonlinearity}\label{S:assumok}

We show that Assumptions~\ref{A:NfNg1}, \ref{A:fgV}, and~\ref{A:N-dif} are satisfied by the double-well polynomial potential~$F$ in~\eqref{polyFtau} and by some nonlinearities~$g$.

%%%%%%%%%%%%%%%%%%%%%%%%%%%%
%%%%%%%%%%%%%%%%%%%%%%%%%%%%
\subsection{The double-well potential}
We start by checking Assumption~\ref{A:fgV}. By~\eqref{polyFtau} we see that~$F$ maps continuous functions to continuous functions. In particular, it maps~$\rmD(A)\subseteq\clC(\overline\Omega)$ into~$L^\infty\subset L^1$; recall that~$\Omega\subset\bbR^d$ is bounded and~$d\in\{1,2,3\}$.

By~\eqref{polyFtau} and the Sobolev embedding~$V\xhookrightarrow{}L^6$, we  find
\begin{align}
\norm{f(v)}{V'}^2&\le C_5\norm{f(v)}{L^\frac65}^2=16\tau^2C_5\norm{v^3-v}{L^\frac65}^2
=32\tau^2C_5\left(\norm{v^3}{L^\frac65}^2+\norm{v}{L^\frac65}^2\right)\notag\\
&\le\tau^2C_5\norm{v}{L^\frac{18}{5}}^{\frac{10}{6}\frac{18}{5}}+\tau^2C_6\norm{v}{L^2}^2
=\tau^2C_7\norm{v}{W^{\frac23,2}}^{6}+\tau^2C_6\norm{v}{H}^2,
\notag
\end{align}
where we used the Sobolev embedding~$W^{\frac23,2}\xhookrightarrow{}L^\frac{18}{5}$. Now, an interpolation argument gives
\begin{align}
\norm{f(v)}{V'}^2&\le\tau^2C_8\norm{v}{H}^{2}\norm{v}{V}^4+\tau^2C_6\norm{v}{H}^2.
\label{chk-assumf1}
\end{align}

Further, under either of the boundary conditions in~\eqref{bcs},  we have
\begin{align}
-\langle A_0 f(v),v\rangle_{\rmD(A)',\rmD(A)}&=-4\tau(\nabla (v^3-v),\nabla v)_{(L^2)^d}=4\tau\int_\Omega (1-3v^2)\norm{\nabla v}{\bbR^d}^2\,\rmd\Omega\notag\\
&\le4\tau\norm{\nabla v}{  (L^2)^d}^2.   \label{chk-assumf2}
\end{align}
By~\eqref{chk-assumf1} and~\eqref{chk-assumf2}, it follows that the nonlinearity in~\eqref{polyFtau} satisfies the requirements in Assumption~\ref{A:fgV}. 
Next we check Assumptions~\ref{A:NfNg1} and~\ref{A:N-dif}.
We find
\begin{equation}\notag
f(v)-f(w)\coloneqq4\tau(d^3+3wd^2+3w^2d-d),
\end{equation}
with~$d\coloneqq v-w$. Here, we make the following assumption on the target trajectory.
\begin{assumption}\label{A:yr-bdd}
We have that~$\norm{y_{1\ttr}}{L^\infty(\bbR_+\times\Omega)}+\norm{\nabla y_{1\ttr}}{L^\infty(\bbR_+\times\Omega)^d}\le C_\ttr<+\infty$.
\end{assumption}

Again, under either of the boundary conditions in~\eqref{bcs}, for~$N_f$ as in~\eqref{sys-z} we obtain
\begin{align}
\clZ&\coloneqq-\langle A_0 N_f(t,v),v\rangle_{\rmD(A)',\rmD(A)}\notag\\
&=-4\tau\int_\Omega(3v^2+6y_{1\ttr}v+3y_{1\ttr}^2-1)\norm{\nabla v}{\bbR^d}^2\,\rmd \Omega -4\tau\int_\Omega(3v^2+6y_{1\ttr}v)(\nabla y_{1\ttr} ,\nabla v)_{\bbR^d}\,\rmd \Omega\notag\\
&=4\tau\int_\Omega(1-3(v+y_{1\ttr})^2)\norm{\nabla v}{\bbR^d}^2\,\rmd\Omega-4\tau\int_\Omega(3v^2+6y_{1\ttr}v)(\nabla y_{1\ttr} ,\nabla v)_{\bbR^d}\,\rmd \Omega\notag\\
&\le4\tau\norm{v}{V_0}^2+4\tau\int_\Omega(3C_\ttr v^2+6C_\ttr^2v)\norm{\nabla v}{\bbR^d}\,\rmd \Omega,\notag
\end{align}
with~$C_\ttr$ as in Assumption~\ref{A:yr-bdd}. Hence,
\begin{align}
\clZ&\le4\tau\norm{v}{V_0}^2+12\tau C_\ttr \norm{v}{L^4}^2\norm{v}{V_0}+24\tau C_\ttr^2 \norm{v}{L^2}\norm{v}{V_0}\notag\\
&\le4\tau\norm{v}{V_0}^2+\tau C_\ttr C_0 \norm{v}{W^{\frac34,2}}^{2}\norm{v}{V_0}+24\tau C_\ttr^2 \norm{v}{H}\norm{v}{V_0},\notag
\end{align}
where we have used~$W^{\frac34,2}(\Omega)\xhookrightarrow{}L^4(\Omega)$. Now, by interpolation inequalities, we find
\begin{align}
\clZ&\le\tau C_1\norm{v}{H}\norm{v}{\rmD(A)}+\tau C_\ttr C_1 \norm{v}{H}^{\frac12}\norm{v}{V}^{\frac52}+\tau C_\ttr^2C_1 \norm{v}{H}\norm{v}{\rmD(A)}\notag\\
&\le\tau C_1\norm{v}{H}\norm{v}{\rmD(A)}+\tau C_\ttr C_2 \norm{v}{H}^{\frac74}\norm{v}{\rmD(A)}^{\frac54}+\tau C_\ttr^2C_1 \norm{v}{H}\norm{v}{\rmD(A)}\notag\\
&\le\tau C_3(1+C_\ttr^2)\left(\norm{v}{H}\norm{v}{\rmD(A)}+ \norm{v}{H}^{\frac74}\norm{v}{\rmD(A)}^{\frac54}\right).\notag
\end{align}

Hence, $N_f$ satisfies the requirements in Assumption~\ref{A:NfNg1}. Next, with~$d\coloneqq v-w$, we find
\begin{align}
&-\langle A_0 (f(v)-f(w)),z\rangle_{V',V}\notag\\
&\qquad=-4\tau\int_\Omega(3d^2+6wd+3w^2-1) (\nabla d,\nabla z)_{\bbR^d}\,\rmd \Omega-4\tau\int_\Omega(3d^2+6wd)(\nabla w ,\nabla z)_{\bbR^d}\,\rmd \Omega,\notag
\end{align}
and, by H\"older and Sobolev inequalities (recall that~$d\in\{1,2,3\}$)
\begin{align}
\norm{A_0 (f(v)-f(w))}{V'}
&\le 4\tau\norm{(1-3(d+w)^2)\nabla d+(3d^2+6wd)\nabla w}{(L^2)^d}\notag\\
&\le 4\tau\norm{1-3v^2}{L^\infty }\norm{d}{V}+12\tau\norm{d^2+2wd}{L^3}\norm{\nabla w}{(L^6)^d} \notag\\
&\le 4\tau\norm{1-3v^2}{L^\infty }\norm{d}{V}+\tau C_4(\norm{d}{L^6}^2+\norm{d}{L^6}\norm{w}{L^6})\norm{w}{\rmD(A)} \notag\\
&\le 4\tau\norm{1-3v^2}{L^\infty }\norm{d}{V}+\tau C_5(\norm{d}{V}+\norm{w}{L^6})\norm{w}{\rmD(A)}\norm{d}{V}, \notag
\end{align}
and, using also the Agmon inequaly ($d\in\{1,2,3\}$)
\begin{align}
\norm{A_0 (f(v)-f(w))}{V'}&\le \tau C_6(1+\norm{v}{V}\norm{v}{\rmD(A)}+\norm{d}{V}\norm{w}{\rmD(A)}+\norm{w}{V}\norm{w}{\rmD(A)})\norm{d}{V} \notag\\
&\le\tau C_7(1+\norm{v}{V}+\norm{w}{V})(1+\norm{v}{\rmD(A)}+\norm{w}{\rmD(A)})\norm{d}{V} \notag\\
&\le\tau C_8(1+\norm{(v,w)}{V\times V})(1+\norm{(v,w)}{\rmD(A)\times\rmD(A)})\norm{d}{V}. \notag
\end{align}
In particular, ~$f$ satisfies the inequality in Assumption~\ref{A:N-dif}.
\begin{remark}\label{R:equil}
We would like to mention that Assumption~\ref{A:yr-bdd} is nonempty. For example, it is satisfied by equilibria. In fact, from the results of the existence and uniqueness in Sections~\ref{sS:exist} and~\ref{sS:unique} we can conclude that every equilibrium~$(\varphi,\phi)\in V\times H$ is necessarily in~$\rmD(A^\frac32)\times V$. The component~$y_{\ttr1}(t)=\varphi$ satisfies the requirements in Assumption~\ref{A:yr-bdd}, due to~$\rmD(A^\frac32)\subset W^{3,2}(\Omega)$, to the Agmon inequality, and to~$d\in\{1,2,3\}$.
\end{remark}

%%%%%%%%%%%%%%%%%%%%%%%%%%%%
%%%%%%%%%%%%%%%%%%%%%%%%%%%%
\subsection{The nonlinearity component~$g$}
We consider examples of a nonlinearities~$g=g(z_1)$ satisfying the requirements in Assumptions~\ref{A:NfNg1}, \ref{A:fgV}, \ref{A:N-dif}. That is, we consider the family
\begin{equation}\label{g-class}
g(w)\coloneqq a_0+a_1 w+a_2w(\norm{w}{\bbR}-1);
\end{equation}
defined by three given scalar functions~$a_0\in L^2(\Omega)$; ~$a_1\in L^\infty(\Omega)$ and~$a_1>0$; $a_2\in L^\infty(\Omega)$ and~$a_2>0$. This class of nonlinearities is motivated by examples given in~\cite{Miranville19}. Indeed, the example in~\cite[Intro., exa.~(i)]{Miranville19} concerns the Cahn--Hilliard--Oono equation with~$g(w)=a_0+a_1 w$ with constants~$a_0$ and~$a_1>0$; in~\cite[Intro., exa.~(ii)]{Miranville19} it is considered the more general ``fidelity term''~$g(w)=a_0+a_1 w$ where~$a_0\in L^2(\Omega)$ and~$0<a_1\in L^\infty(\Omega)$, motivated by applications to binary image inpainting; finally, \cite[Intro., exa.~(iii)]{Miranville19} considers the ``proliferation term''~$g(w)=a_2w(w-1)$ with constant~$a_2>0$, motivated by applications in biology; above we need a variation of this term because our Assumption~\ref{A:fgV} implicitly requires a monotonicity property (away from~$0$).

To show that our results cover nonlinearities~$g$ as in~\eqref{g-class}, we start by observing that the term~$a_0$ can be included in the external forcing, that is, by taking~$h_1+a_0$ instead of~$h_1$ in~\eqref{sys-CH}. Then it is enough that, the assumptions are satisfied by
\begin{equation}\label{barg-class}
\overline g(w)\coloneqq a_1 w+a_2w(\norm{w}{\bbR}-1).
\end{equation}
We find, that for a suitable constant~$C_1\ge 0$,
\begin{align}
\langle\overline g(w),v\rangle_{V,V'}&\le C_1\left(\norm{w}{V'}+\norm{w^2}{V'}\right)\norm{v}{V}\le C_2\left(\norm{w}{H}+\norm{w}{L^4}^2\right)\norm{v}{V}\notag\\
&\le C_3\left(\norm{w}{H}+\norm{w}{H}^\frac12\norm{w}{V}^\frac32\right)\norm{v}{V}\le C_4\norm{w}{V}\left(1+\norm{w}{H}^\frac12\norm{w}{V}^\frac12\right)\norm{v}{V}\notag\\
&\le C_5\norm{w}{V}\left(1+\norm{w}{H}\norm{w}{V}\right)\norm{v}{V}.\notag
\end{align}
and~$-\langle\overline g(w),w\rangle_{\rmD(A),\rmD(A)'}\le\norm{a_2}{L^\infty}\norm{w}{H}^2$. 
Hence, Assumption~\ref{A:fgV} is satisfied by~$\overline g$.    

Next we observe that, with~$d=w-v$,
\begin{align}
\overline g(w)-\overline g(v)&=  (a_1-a_2) d+a_2d\norm{w}{\bbR}+a_2v(\norm{w}{\bbR}-\norm{v}{\bbR}).\label{g-aux-ex}\\
\norm{\overline g(w)-\overline g(v)}{V'}&\le C_6\norm{d}{H}\left(1+ \norm{w}{L^\infty}+\norm{v}{L^\infty}\right)\le C_7\left(1+ \norm{(w,v)}{\rmD(A)\times\rmD(A)}\right)\norm{d}{V},\notag
\end{align}
which gives us that Assumption~\ref{A:N-dif} is satisfied by~$\overline g$.   Finally, replacing~ $v$ and $w$ in~\eqref{g-aux-ex}, with the reference trajectory component~$y_{\ttr1}$  and $y_{\ttr1} +d$,  respectively,  we find
\begin{align}
-\langle N_g(d), d\rangle_{V',V}&=-\langle \overline g(y_{\ttr1}+d)-\overline g(y_{\ttr1}), d\rangle_{V',V}\notag\\
&= -\left( a_1-a_2+a_2\norm{y_{\ttr1}+d}{\bbR}, d^2\right)_H-\left(a_2y_{\ttr1}(\norm{y_{\ttr1}+d}{\bbR}-\norm{y_{\ttr1}}{\bbR}),d\right)_H\notag\\
&\le\norm{a_2}{L^\infty}\norm{d}{H}^2+\norm{a_2y_{\ttr1}}{L^\infty}\norm{d}{H}^2.\notag
\end{align}
Hence, Assumption~\ref{A:NfNg1} is satisfied by~$\overline g$, for essentially bounded target trajectories~$y_{\ttr1}$, in particular, for target trajectories as in Assumption~\ref{A:yr-bdd}.

%%%%%%%%%%%%%%%%%%%%%%%%%%%%
%%%%%%%%%%%%%%%%%%%%%%%%%%%%
%%%%%%%%%%%%%%%%%%%%%%%%%%%%
\section{Numerical Discretization and Simulations}\label{S:simul}

In this section, we present numerical illustrations to validate  our theoretical findings. The numerical results discussed herein have been obtained by implementing the model using the finite element library FEniCS \cite{fenics}.

\subsection{Discretization}

We aim to investigate the coupled system \eqref{sys-z} in terms of both the free dynamic solution~$(z_{1\ttr}, z_{2\ttr}, \varpi_\ttr)$ and the controlled dynamics solution~$(z_{1\ttc}, z_{2\ttc}, \varpi_\ttc)$. Here,~$\varpi_\ttr$ and~$\varpi_\ttc$ denote the chemical potentials of the respective phase-field variable. All variables are equipped with homogeneous Neumann boundary conditions. Quadrilateral elements are employed to approximate the solutions. In this section we focus on the nonlinearity~$f$ and consider the case~$g=0$ only. The variational system is discretized in time using a semi-implicit Euler method, implementing the classical convex-concave splitting scheme for the nonlinear potential function~$f$. Specifically, for the double-well potential~$\frac14\tau(1-y^2)^2$, we define~$f(y)=\tau y(y^2-1)$, where~$f_e(y) = \tau y(y^2-3)$ and~$f_c(y) = 2\tau y$. This approach, treating the expansive part explicitly and the contractive part implicitly, renders the Cahn--Hilliard equation unconditionally stable \cite{elliott1993global}. 

The oblique projection is handled implicitly/explicitly. We follow the discretization approach in~\cite{RodSturm20} for the oblique projections, and present a discretization for the feedback operators. We start with the feedback operator in the Cahn--Hilliard equation, and utilize the following expression \cite[Eq. (3.2)]{RodSturm20} (see also~\cite[Eq. (2.3a)]{KunRod19-cocv}):
\[ P_{\widetilde\clU_{M}}^{\clU_{M}^\perp} y(x) = \sum_{i=1}^{M_\sigma} \alpha_j \widetilde\Phi_j^M(x), \quad \text{with} \quad  \alpha \coloneqq [(U_M,\widetilde U_M)_H]^{-1}[(U_M,y)_H] \in \mathbb{R}^{M_\sigma}. \]

For the convenience of the reader, we provide a brief overview of the discretization approach. We define the actuators matrix~$[U_M]= \left[\bm\Phi_1^M~\cdots~\bm\Phi_{M_\sigma}^M\right]\in \mathbb R^{N \times M_\sigma}$, where the columns are the vectors of the evaluations of the indicator-function actuators  at the~$N$ mesh points in the domain~$\Omega$. Similarly, we denote the auxiliary functions matrix~$[\widetilde U_M]= \left[\widetilde{\bm\Phi}_1^M~\cdots~\widetilde{\bm\Phi}_{M_\sigma}^M\right]\in \mathbb R^{N \times M_\sigma}$. In addition, we consider the mass matrix~$\bm{M}\in \mathbb{R}^{N \times N}$ and the mesh-evaluated function~$\bm y=[y] \in  \mathbb{R}^{N\times1}$. Then, we obtain theapproximations
\[
\alpha \approx[\bm{\Phi}^\top \bm M \bm{\widetilde\Phi}]^{-1} [U_M]^\top \bm M \bm y\quad\mbox{and}\quad P_{\widetilde\clU_{M}}^{\clU_{M}^\perp} y \approx[\widetilde U_M] [\bm{\Phi}^\top \bm M \bm{\widetilde\Phi}]^{-1} [U_M]^\top \bm M \bm y.
\]
Similarly we obtain the approximation
\[
 P_{\clU_{M}}^{\widetilde\clU_{M}^\perp} y \approx[U_M] [\widetilde{\bm\Phi}^\top \bm M \bm{\Phi}]^{-1} [\widetilde U_M]^\top \bm M \bm y.
\]
Denoting, for simplicity,~$\bfQ_1\coloneqq[{\bm\Phi}^\top \bm M \widetilde{\bm\Phi}]^{-1}$, we find
\[
(P_{\widetilde\clU_{M}}^{\clU_{M}^\perp} y, \xi)_H \approx\bm\xi^\top \bm M [\widetilde U_M] \bm Q_1 \bm [ U_M]^\top \bm M \bm y\quad\mbox{and}\quad
(P_{\clU_{M}}^{\widetilde\clU_{M}^\perp} y, \xi)_H \approx\bm\xi^\top \bm M [U_M]  \bm Q_1^\top[\widetilde U_M]^\top \bm M \bm y.
\]
which in particular agrees with the fact that~$P_{\widetilde\clU_{M}}^{\clU_{M}^\perp}=(P_{\clU_{M}}^{\widetilde\clU_{M}^\perp} )^*$, namely,
 \begin{align}
\bm M [\widetilde U_M] \bm Q_1 \bm [ U_M]^\top \bm M =\left(\bm M [U_M]  \bm Q_1^\top[\widetilde U_M]^\top \bm M\right)^\perp.
\end{align}
Next, we obtain the discretization of the feedback operator~$\bfF_1$ in~\eqref{Kmonotone} as follows. First of all, for~$A$ as in Section~\ref{sS:fun-sett}, we find~$A_D=\bm S+\bm M$ where~$\bm S$ is the stiffness matrix, then we observe that
 \begin{align}
(\bfF_1 y, \xi)_H &=-\lambda_1(AP_{\widetilde\clU_{M}}^{\clU_{M}^\perp} y, AP_{\widetilde\clU_{M}}^{\clU_{M}^\perp} \xi)_H \approx-\lambda_1[AP_{\widetilde\clU_{M}}^{\clU_{M}^\perp} \xi]^\perp\bm M [AP_{\widetilde\clU_{M}}^{\clU_{M}^\perp} y]\notag\\
&=-\lambda_1\left(\bm M^{-1}\bm S [ P_{\widetilde\clU_{M}}^{\clU_{M}^\perp} \xi]\right)^\perp\bm M \left(\bm M^{-1}\bfS [P_{\widetilde\clU_{M}}^{\clU_{M}^\perp} y]\right)\notag\\
&=-\lambda_1 [ P_{\widetilde\clU_{M}}^{\clU_{M}^\perp} \xi]^\perp \bm S\bm M^{-1}\bm S [P_{\widetilde\clU_{M}}^{\clU_{M}^\perp} y].\notag
\end{align}
We can follows analogous arguments to obtain
\begin{align}
(\bfF_2 y, \xi)_H &=-\lambda_2\langle AP_{\widetilde\clV_{M}}^{\clV_{M}^\perp} y, P_{\widetilde\clV_{M}}^{\clV_{M}^\perp} \xi\rangle_{V',V} \approx-\lambda_2[P_{\widetilde\clV_{M}}^{\clV_{M}^\perp} \xi]^\perp\bm S [P_{\widetilde\clV_{M}}^{\clV_{M}^\perp} y].\notag
\end{align}
Next, we use~$P_{\widetilde\clU_{M}}^{\clU_{M}^\perp} y\approx[\widetilde U_M] \bm Q_1 \bm [ U_M]^\top \bm M \bm \xi$ and~$P_{\widetilde\clV_{M}}^{\clV_{M}^\perp} y\approx[\widetilde V_M] \bm Q_2 \bm [ V_M]^\top \bm M \bm \xi$, with~$\bfQ_2\coloneqq[{\bm\Psi}^\top \bm M \widetilde{\bm\Psi}]^{-1}$ and find
\begin{align}
(\bfF_1 y, \xi)_H
&\approx-\lambda_1 \left([\widetilde U_M] \bm Q_1 \bm [ U_M]^\top \bm M \bm \xi\right)^\perp \bm S\bm M^{-1}\bm S [\widetilde U_M] \bm Q_1 \bm [ U_M]^\top \bm M \bm y\notag\\
&=-\lambda_1 \bm \xi^\perp\bm M\left([\widetilde U_M] \bm Q_1 \bm [ U_M]^\top  \right)^\perp \bm S\bm M^{-1}\bm S [\widetilde U_M] \bm Q_1 \bm [ U_M]^\top \bm M \bm y\notag\\
&=-\lambda_1 \bm \xi^\perp\bm M[ U_M]\bm Q_1^\perp\left([\widetilde U_M]^\perp \bm S\bm M^{-1}\bm S [\widetilde U_M]\right) \bm Q_1 \bm [ U_M]^\top \bm M \bm y;\notag\\
(\bfF_2 y, \xi)_H &\approx-\lambda_2\left([\widetilde V_M] \bm Q_2 \bm [ V_M]^\top \bm M \bm \xi\right)^\perp \bm S [\widetilde V_M] \bm Q_2 \bm [ V_M]^\top \bm M \bm y\notag\\
&=-\lambda_2 \bm \xi^\perp\bm M[ V_M]\bm Q_2^\perp\left([\widetilde V_M]^\perp \bm S [\widetilde V_M]\right) \bm Q_2 \bm [ V_M]^\top \bm M \bm y.\notag
\end{align}
That is, the discretizations of~$\bfF_1$ and~$\bfF_2$ are taken as
\begin{align}
&(\bfF_1)_D\coloneqq \bm M[ U_M]\widehat \bfF_1\quad\mbox{and}\quad(\bfF_2)_D\coloneqq \bm M[ V_M]\widehat \bfF_2,\\
\mbox{with}\qquad
&\widehat \bfF_1\coloneqq-\lambda_1\bm Q_1^\perp \bfR_1 \bm Q_1 \bm [ U_M]^\top \bm M\in\bbR^{M_\sigma\times N},\qquad\bfR_1 \coloneqq [\widetilde U_M]^\perp \bm S\bm M^{-1}\bm S [\widetilde U_M],\notag\\
&\widehat \bfF_2\coloneqq -\lambda_2\bm Q_2^\perp \bfR_2 \bm Q_2 \bm [ V_M]^\top \bm M\in\bbR^{M_\varsigma\times N},\qquad\bfR_2 \coloneqq [\widetilde V_M]^\perp \bm S [\widetilde V_M].
\end{align}

We precompute and store the feedback matrices~$\widehat \bfF_1$ and~$\widehat \bfF_2$ mapping the state into the coordinates of the inputs. We can We precompute and store the feedback matrices~$\widehat \bfF_1$ and~$\widehat \bfF_2$ mapping the state into the coordinates of the inputs. Note that we do not need to compute~$\bm M^{-1}$ to construct~$\bfR_1$, because we can perform the computations as~$\bfR_1= [\widetilde U_M]^\perp (\bm S(\bm M^{-1}(\bm S [\widetilde U_M])))$. Indeed, in this way, once we have computed~$\zeta\coloneqq\bm S [\widetilde U_M]\in\bbR^{N\times M_\sigma}$, then the (columns of the) matrix
$\bm M^{-1}\zeta\in\bbR^{N\times M_\sigma}$ can be computed iteratively; note that~$\bm M$ is symmetric and positive definite. Finally, recall that the number~$M_\sigma+M_\varsigma$ of actuators in~$U_M\cup V_M$ is independent of (and much smaller than) the number~$N$ of mesh points.

\subsection{Simulation Setup}

We consider a unit square domain~$\Omega = [0, 1]^2$ and a time domain~$[0,T]$ with~$T=1$, discretized with~$\Delta x = 1/150$ and~$\Delta t=5 \cdot 10^{-5}$. Setting~$\nu_2=0.005$ and~$\tau=2$ ensures that the phase-field's interface can be resolved, satisfying the condition~$0.0067 \approx \Delta x \ll  \sqrt{\nu_2/\nu_3}=0.05$, see \cite{deckelnick2005computation}. For the time-step size, we ensure the condition for energy stability in the fully implicit scheme, see \cite{barrett1999finite},~$\Delta t < 4\nu_2/\nu_3$, with~$\Delta t<0.01$. Smaller time steps are necessary for larger values of~$\lambda_1$ and~$\lambda_2$.

The parameters for all simulations are chosen as~$\nu_0=0.1$, $\nu_1=1$, $\nu_2=0.005$, and $\nu_3=2$. Additionally, for simplicity, we focus on the case~$g=0=N_g$; see system~\eqref{sys-z}.

The total volume covered by the used~$M_{\rm tot}\coloneqq M_\sigma+M_\varsigma=(2M)^2$ actuators is 6.25\% of the whole domain, independently of~$M$. We vary the parameters~$\lambda_1$ and~$\lambda_2$ in the simulation scenarios, setting, for simplicity,
\begin{equation}\notag
\lambda_1=\lambda_2\eqqcolon \lambda \in \{50,100,250,500,750,1000\}.
\end{equation}
The total number~$M_{\rm tot}$ of actuators varies between examples, with~$M\in\{1,2,3\}$ corresponding to~$M_{\rm tot}\in\{4,16,36\}$, as shown in Fig.~\ref{Fig:Actuators} (cf.~Fig.~\ref{fig.suppActSqu}).
\begin{figure}[H]
    \centering
\includegraphics[width=.99\textwidth,page=1]{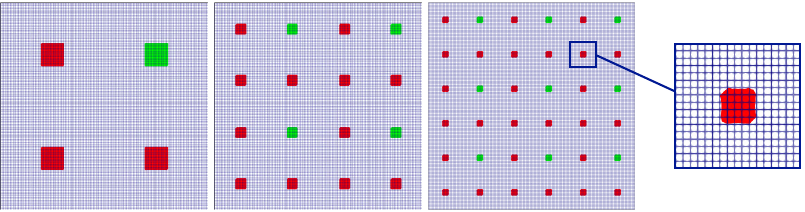}
       \caption{Depiction of actuators for~$M=1$ (left),~$M=2$ (middle) and~$M=3$ (right) on the spatial grid; the red actuators influence the~$z_{1\ttc}$ equation whereas the green actuators influence the~$z_{2\ttc }$ equation.}
    \label{Fig:Actuators}
\end{figure}

%%%%%%%%%%%%%%%%%%%%%%%
%%%%%%%%%%%%%%%%%%%%%%%
\subsection{Simulation results}
The target trajectory is taken as the solution corresponding to the initial state
\begin{equation}\label{ini-r}
z_{1\ttr ,0}(x)=\tanh(100(x_1-\tfrac12)), \qquad z_{2\ttr ,0}(x)=1.
\end{equation}

The time-snapshots in Fig.~\ref{Fig:Initial},
\begin{figure}[ht]
	\includegraphics[page=11,width=.90\textwidth]{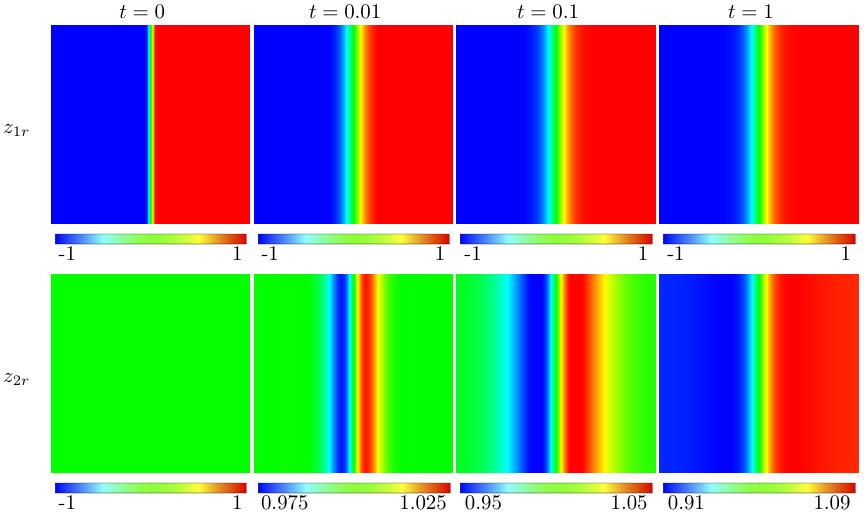}
	\caption{Evolution of the reference~$z_\ttr=(z_{1\ttr},z_{2\ttr})$ with initial state~\eqref{ini-r}.}
    \label{Fig:Initial}
\end{figure}
in the first row, show that the interface of $z_{1r}$ is initially growing and converges asymptotically to an equilibrium (the state changes after~$t=0.01$ are imperceptible, with the naked eye) with separated phases, as shown by the order-parameter/phase-field component~$y_{1\ttr }=z_{1\ttr }$ separated into values of~$-1$ and~$1$, with a smooth transition zone between. In addition, we observe that $z_{2r}$ evolves from a constant initial condition to form an interface at the same position as $z_{1r}$, driven by the coupling terms in the system. 
 Next, we consider another initial state as
 \begin{equation}\label{ini-c}
z_{1\ttc,0}(x)=0.3+\frac{\cos(2\pi x_1) \cos(2\pi x_2)}{100}, \qquad z_{2\ttc ,0}(x)=0.
\end{equation}
Firstly, we observe that to track the reference target trajectory~$y_{1\ttr }$ we will need a control forcing. Indeed, in Fig.~\ref{Fig:EquilibriumCos}
 \begin{figure}[ht]
 	    \centering \includegraphics[width=.90\textwidth,page=2]{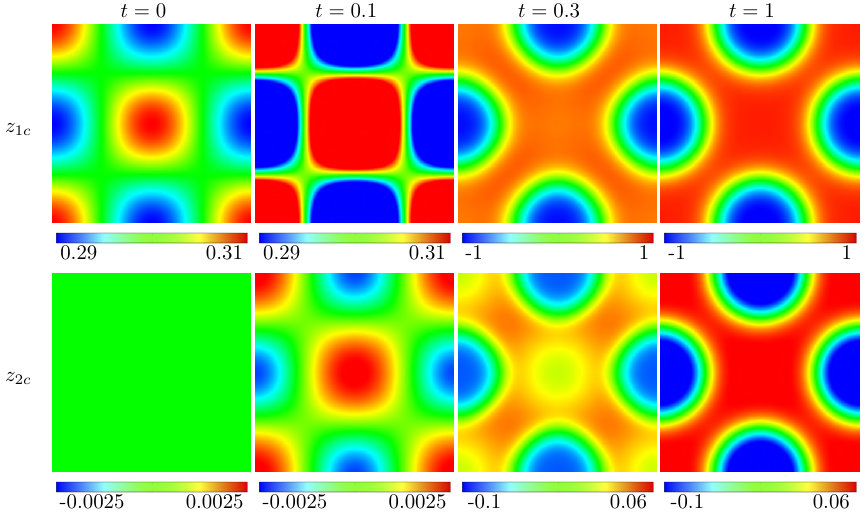}
    \caption{Case~$\lambda=0$.  Free dynamics  evolution of~$z_{\ttc}$, with initial state~\eqref{ini-c}.} \label{Fig:EquilibriumCos}
\end{figure}
we depict the free-dynamics solution (i.e.,  without control forcing,  corresponding to taking~$\lambda=0$), issued from the initial state~\eqref{ini-c}; we see that it
evolves to a different equilibrium with separated phases, again with a smooth transition zone in between. The dynamics of temperature~$z_{2\ttc }$ is significantly influenced by the phase-field variable~$z_{1\ttc}$ due to the coupling introduced by~$\nu_1$, with a similar structure noted in~$z_{2\ttc }$ but with smaller values  since the initial value was set to zero.

Due to Theorem~\ref{T:main}, we know that we have stabilizability for~$\lambda$ and~$M$ large enough.

%%%%%%%%%%%%%%%%%%%%%%%%
\subsubsection{Evolution of the solutions through time-snapshots}
We start by taking~$M=2$ and investigate the behavior for several values of~$\lambda$.
The convergence of~$z_{\ttc}$ to the target~$z_{\ttr}$ is checked
 in Fig.~\ref{Fig:EvolutionZ1Feedback} for $z_{1c}$ and Fig.~\ref{Fig:EvolutionZ2Feedback} for $z_{2\ttc}$,  for large values of~$\lambda\in\{500,1000\}$, the component~$z_{1\ttc}$ of the solution at time~$t=1$ is already close to the reference targeted one~$z_{1\ttr}$, in particular, with a distinct interface close to the vertical one as in Fig.~\ref{Fig:Initial}.
  \begin{figure}[ht]
 	    \centering \includegraphics[width=.95\textwidth,page=4]{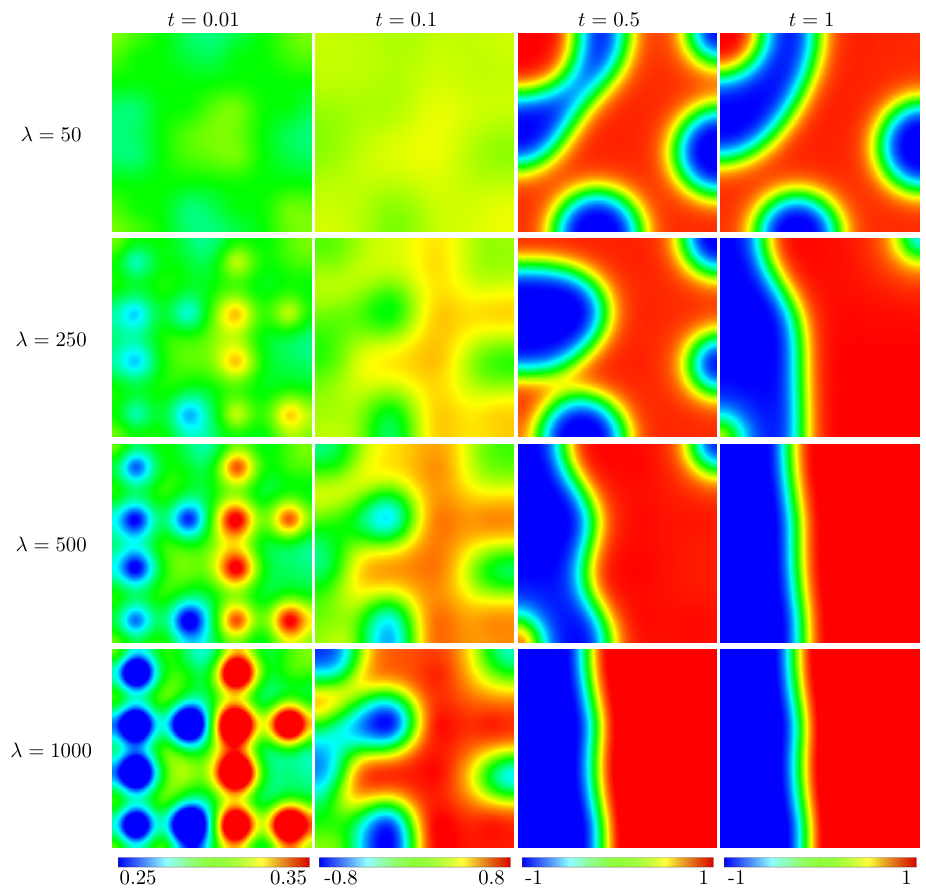}
    \caption{Case~$M=2$ with varying~$\lambda$. Evolution of~$z_{1\ttc}$ with feedback control.}
    \label{Fig:EvolutionZ1Feedback}
\end{figure}
\begin{figure}	
 	    \centering \includegraphics[width=.99\textwidth,page=12]{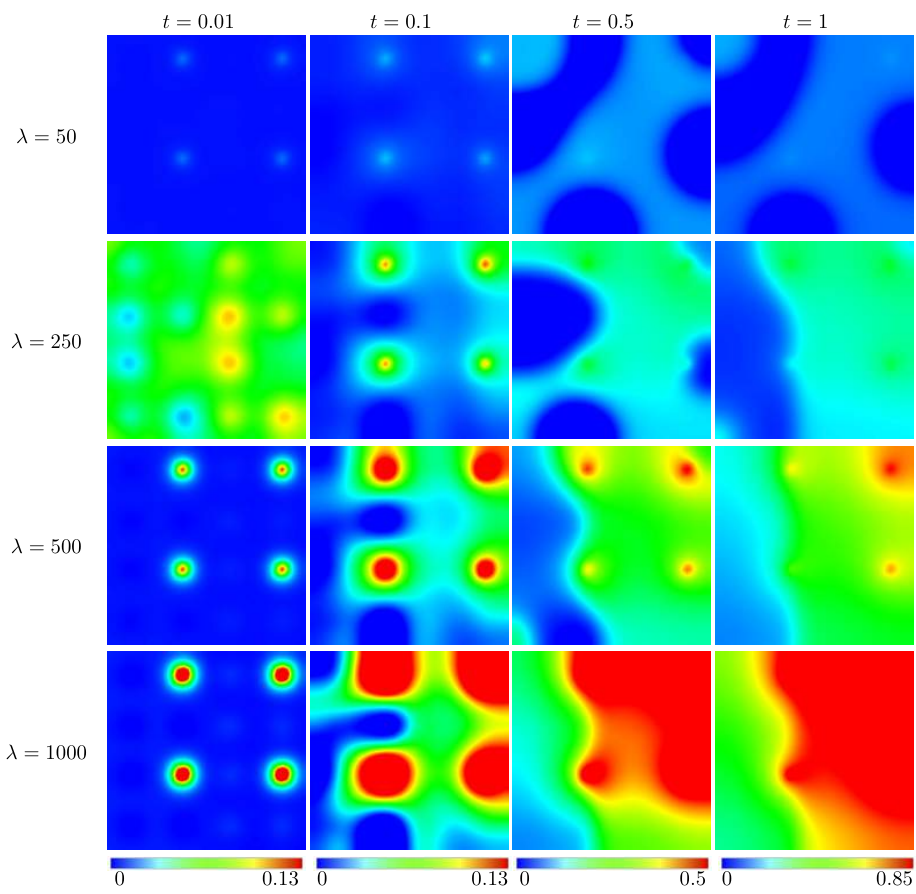}
    \caption{Case~$M=2$ with varying~$\lambda$. Evolution of~$z_{2\ttc}$ with feedback control.}
    \label{Fig:EvolutionZ2Feedback}
\end{figure}
In addition, in the time-snapshot at small time~$t=0.01$, we can observe a pronounced influence of the control forcing on the state at the placement of the~$16=M_{\rm tot}=(2M)^2$ actuators (cf. Fig.~\ref{Fig:Actuators}). We observe that the reference trajectory determines the value within the actuator, which then diffuses outward across the domain. A similar behavior is observable in Fig.~\ref{Fig:EvolutionZ2Feedback} for $z_{2c}$, where again the action of the control forcing reveals the location of the four actuators acting on the heat equation.

%%%%%%%%%%%%%%%%%%%%%%%%
\subsubsection{Evolution of the norm of the distance to the target}
We proceed by comparing logarithmic error plots between~$z_r$ and~$z_c$ to varying~$\lambda \in \{50,100,250,500,750,1000\}$ and~$M \in \{1,2,3\}$. These plots reflect the stability result that we have analyzed in Theorem~\ref{T:main}. We see in Fig.~\ref{Fig:VaryTog}
\begin{figure}[ht]
    \centering
    \includegraphics[height=.245\textheight,page=9]{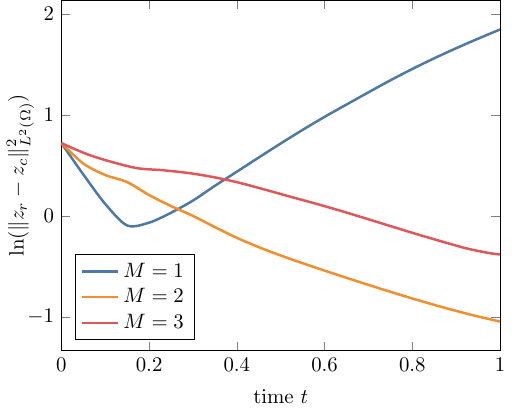}\includegraphics[height=.245\textheight,page=10]{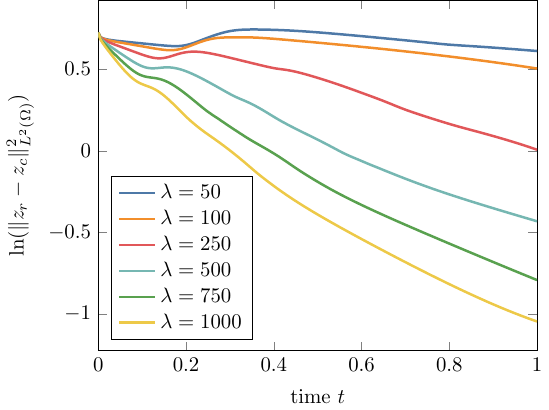}
    \caption{Norm of~$z_r-z_c$. For varying~$M$ with fixed~$\lambda=1000$ (left) and for varying~$\lambda$ with fixed~$M=2$ (right).}
    \label{Fig:VaryTog}
\end{figure}
 that, for smaller values of~$M$ or~$\lambda$, the norm of the state is not strictly decreasing as in Theorem \ref{T:main},  namely, for the case~$(M,\lambda)=(1,1000)$ and the cases~$\lambda\le500$ with~$M=2$. However, we observe this effect for the other cases, validating our theoretical result. In particular, we observe that, for fixed~$M$, the decay is faster for larger~$\lambda$. For fixed~$\lambda=1000$, both $M=2$ and $M=3$ decay at a similar rate.

We go into more detail and plot the norms of the coordinates~$z_{ir}-z_{ic}$, $i\in\{1,2\}$, of the difference~$z_{r}-z_{c}$, to highlight the effects on each equation. In particular, in Fig.~\ref{Fig:VaryActNum},
\begin{figure}[ht]
    \centering
\includegraphics[height=.245\textheight,page=5]{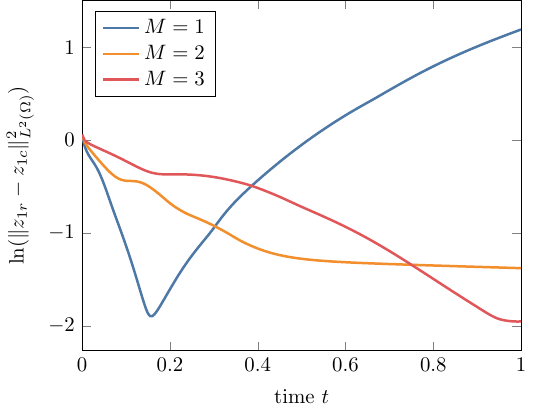} \includegraphics[height=.245\textheight,page=6]{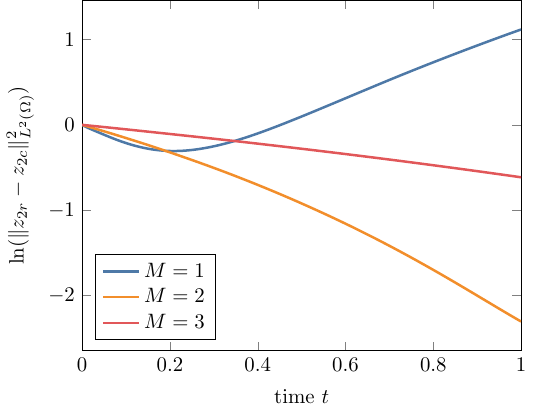}
    \caption{Norm of the coordinates of~$z_r-z_c$ for  varying~$M$ and fixed~$\lambda=1000$.}
    \label{Fig:VaryActNum}
\end{figure}
 we observe a large initial drop for $M=1$ and fixed value~$\lambda=1000$ and then an increase in error afterward, which shows that the~$4$ actuators in this case are not enough for the controlled solution~$z_{\ttc}$ to track the targeted one~$z_{\ttr}$. By increasing the number of actuators as in the cases~$M\in\{2,3\}$ we see that~$z_{\ttc}$ is able to track~$z_{\ttr}$. the In Fig.~\ref{Fig:VaryLamVal}
\begin{figure}[ht]
    \centering
\includegraphics[height=.24\textheight,page=7]{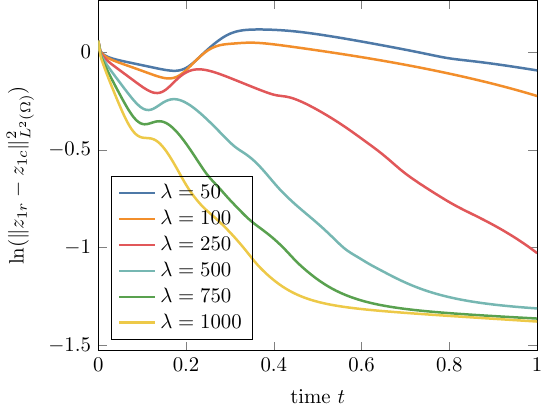}
\includegraphics[height=.24\textheight,page=8]{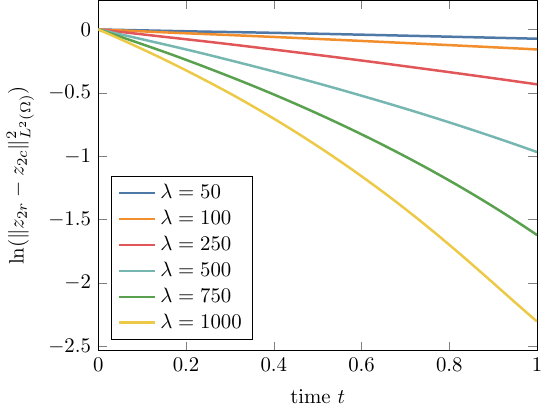}
    \caption{Norm of the coordinates of~$z_r-z_c$  for varying~$\lambda$ and fixed~$M=2$.}
    \label{Fig:VaryLamVal}
\end{figure}
 we fix $M=2$ and see that, in this case, also the decreasing rate of the heat component~$z_{2\ttr}-z_{2\ttc}$ clearly increases with~$\lambda$ agreeing with the results for the pair~$z_{\ttr}-z_{\ttc}$ in Fig.~\ref{Fig:VaryTog}.
The norm evolution of the component~$z_1=z_{1\ttr }-z_{1\ttc}$ exhibits more irregular behavior than~$z_2=z_{2\ttr }-z_{2\ttc}$, likely due to the more direct effects introduced by the nonlinear potential function~$f$; see~\eqref{polyFtau} and~\eqref{sys-z}.

%%%%%%%%%%%%%%%%%%%%%%%%
%%%%%%%%%%%%%%%%%%%%%%%%
%%%%%%%%%%%%%%%%%%%%%%%%
%\newpage
\appendix
\gdef\thesection{\Alph{section}}
\section*{Appendix}\normalsize
\setcounter{section}{1}%to restart with A (~1 in Alph)
\setcounter{theorem}{0} \setcounter{equation}{0}
\numberwithin{equation}{section}

%%%%%%%%%%%%%%%%%%%%%%%%
%%%%%%%%%%%%%%%%%%%%%%%%
\subsection{Proof of Lemma~\ref{L:poincare}}\label{Apx:proofL:poincare}
For simplicity, let~$\Omega=\bfR\subset\bbR^d$ be a rectangular spatial domain. An analogous argument can be followed for general convex polygonal domains~$\Omega=\bfP\subset\bbR^d$; see~\cite[Rem.~2.8]{AzmiKunRod23-ieee}.

The divergence~$\lim_{M\to+\infty}\alpha_{M}^H=+\infty$ is shown in~\cite[Sect.~5]{Rod21-sicon} where the support of the reference indicator function~$\omega_{21}^1$ is located at the center of a given rectangular spatial domain~$\Omega=\bfR\subset\bbR^d$. The steps of the proof can be repeated for an arbitrary location of~$\omega_{21}^1$, for example as in Fig.~\ref{fig.suppActSqu}.

The divergence~$\lim_{M\to+\infty}\alpha_{M}^V=+\infty$ follows from a slight variation of the arguments in~\cite[Sect.~4.2]{Rod21-jnls} for a rectangular spatial domain~$\Omega=\bfR\subset\bbR^d$. Note that in~\cite[Sect.~4.2]{Rod21-jnls}, the space~$\clU_1$ is taken larger with dimension~$d^2$ for simplicity, but this space can be taken with dimension~$d+1$, due to~\cite[Prop.~4.3]{Rod21-jnls}. Note that, denoting by~$c^j$ the center (of mass) of~$\omega_{1j}^1$,~$1\le j\le d+1$,  then for any given polynomial~$p(x)$ of degree~$1$ we find that
\begin{equation}\notag
\int_{\omega_{1j}^1}p(x)\rmd x=0\quad\Longleftrightarrow\quad\int_{\omega_{11}^1}p(x-\tau^j)\rmd x=0
\end{equation}
with~$\tau^j\coloneqq c^j-c^1$. Since~$p(x)=a_0+\sum_{i=1}^d a_ix_i$ is affine, we arrive at the analogue of~\cite[Eq.~(A.31), Sect.~A.7]{Rod21-jnls} as follows,
\begin{align}
&\int_{\omega_{11}^1}p(x-\tau^j)\rmd x=0\quad\mbox{for all}\quad j\in\{1,2,\dots, d+1\}\notag\\
\quad\Longleftrightarrow\quad&\int_{\omega_{11}^1}p(x)+p(\tau^j)-a_0\,\rmd x=0\quad\mbox{for all}\quad j\in\{1,2,\dots, d+1\}.\notag
\end{align}
Hence, by taking the centers~$c^j$ not all in the same hyperplane, it follows that the family~$\{\tau^j\mid 2\le j\le d+1\}$ is linearly independent. In this case, we can conclude that~$\int_{\omega_{1j}^1}p(x)\rmd x=0$ for all~$j\in\{2,\dots, d+1\}$ if, and only if, the linear mapping~$p(x)-a_0$ vanishes. Then, necessarily~$p(x)=a_0$, and~$\int_{\omega_{11}^1}p(x)\rmd x=0$ implies~$p=0$. That is, we can follow the arguments in~\cite[Sect.~A.7]{Rod21-jnls} to obtain the analogue of~\cite[Prop.~4.3]{Rod21-jnls}. This allows us to follow the remaining arguments in~\cite[Sect.~A.7]{Rod21-jnls} to show the that~$\lim_{M\to+\infty}\alpha_{M}^V=+\infty$ for a rectangular spatial domain~$\Omega=\bfR\subset\bbR^d$. 
\qed

%%%%%%%%%%%%%%%%%%%%%%%%
%%%%%%%%%%%%%%%%%%%%%%%%
\subsection{Proof of Lemma~\ref{L:M-and-lam}}\label{Apx:proofL:M-and-lam}
Inequality~\eqref{L:M-and-lamH} is shown in~\cite[Lem.~2.6]{KunRod23-dcds}, using the divergence of~$\alpha_M^H$ in Lemma~\ref{L:poincare}. Similar arguments can be followed to show inequality~\eqref{L:M-and-lamV}, using the divergence of~$\alpha_M^V$ in Lemma~\ref{L:poincare}.

\begin{remark}
Variations of the inequalities~\eqref{L:M-and-lamV} and~\eqref{L:M-and-lamH} in Lemma~\ref{L:M-and-lam} can be shown following the proofs in~\cite[Lem.~3.5]{KunRodWal21} and~\cite[Lem.~3.6]{Rod21-jnls}, which can be useful to prove stabilizing properties of appropriate variations of the feedback operators in~\eqref{OP-FeedK}.
\end{remark}

%\newpage
%\pagebreak
%%%%%%%%%%%%%%%%%%%%%%%%
%%%%%%%%%%%%%%%%%%%%%%%%
\bigskip\noindent
{\bf Aknowlegments.}
S.~S. Rodrigues acknowledges partial support from State of Upper Austria and Austrian Science
Fund (FWF): P 33432-NBL.

%%%%%%%%%%%%%%%%%%%%%%%%
%%%%%%%%%%%%%%%%%%%%%%%%
{\small%
 \bibliographystyle{plainurl}
  \bibliography{CH_OProj}
}

\end{document}